\documentclass{amsart}
\usepackage{tikz}
\usepackage{xcolor}
\usepackage{amssymb,latexsym,amsmath,extarrows}
\usepackage{graphicx,mathrsfs,comment}
\usepackage[margin=1in, centering]{geometry}
\usepackage{hyperref,url}

\usepackage{amstext}
\usepackage{bbm}

\numberwithin{equation}{section}

\newtheorem{theorem}{Theorem}[section]
\newtheorem{lemma}[theorem]{Lemma}

\newtheorem{proposition}[theorem]{Proposition}
\newtheorem{remark}[theorem]{Remark}

\newtheorem{definition}[theorem]{Definition}
\newtheorem{corollary}[theorem]{Corollary}

\usepackage{booktabs}

\newcommand{\R}{\mathbb{R}}

\begin{document}

\title[]{Improved $L^p$ bounds for the helical maximal function in dimensions $n\ge5$}


\author{Changkeun Oh}\address{ Changkeun Oh\\ Department of Mathematical Sciences and RIM, Seoul National University, Republic of Korea} \email{changkeun@snu.ac.kr}

\author{Jaehyun Woo} \address{ Jaehyun Woo \\ Department of Mathematical Sciences, Seoul National University, Republic of Korea}\email{wjhwsh1@snu.ac.kr}

\begin{abstract}
Let $n\ge 5$. We prove that the helical maximal operator is
bounded on $L^p(\mathbb R^n)$ for some $p<2n-4$. This improves the previously known range $p>2n-4$ obtained by
Gan-Maldague-Oh \cite{gan2025sharplocalsmoothingestimates}. As observed by Beltran--Hickman \cite{BH25}, the exponent
$2n-4$ is the natural limit of the standard
local smoothing method.
\end{abstract}

\maketitle


\section{Introduction}

Assume that a curve $\gamma: (-1,1) \rightarrow \R^n$ is smooth and nondegenerate in the sense that
\begin{equation}
 \det{ \big( \gamma'(s),\gamma''(s),\ldots,\gamma^{(n)}(s) \big)} \neq 0
\end{equation}
for all $s \in (-1,1)$. Define the averaging operator 
\begin{equation}
 A_tf(x) = \int_{-1}^1 f \big(x-t \gamma(s) \big) \psi(s)\,ds
\end{equation}
where $\psi\in C_c^\infty((-1,1))$ is a fixed nonzero function. Define the helical maximal operator by
\begin{equation}
 \mathscr Mf(x)=\sup_{t>0} |A_tf(x)|.
\end{equation}

Our main theorem is as follows.

\begin{theorem}\label{0524.thm51}
Let $n \geq 5$. There exists $\epsilon_0=\epsilon_0(n)>0$ such that for every $p>2n-4-\epsilon_0$, we have
\begin{equation}
 \|\mathscr Mf\|_{L^p(\R^n)} \leq C_{p,n} \|f\|_{L^p(\R^n)}
\end{equation}
for some constant $C_{p,n}$ depending only on $p$ and $n$.
\end{theorem}

A standard example shows that $p>n$ is necessary, and this is conjectured to be the sharp range; see the discussion following \cite[Theorem~1.4]{KLOsm}.
For averages over dilates of convex planar curves, Bourgain proved the sharp $L^p$ bound for $p>2$ \cite{Bourgain1986}; Mockenhaupt--Seeger--Sogge later obtained the same range through local smoothing \cite{MockenhauptSeegerSogge1992}. In $\mathbb R^3$, Pramanik--Seeger established the first bounds for sufficiently large $p$ \cite{pramanik2007p}. The sharp range $p>3$ was subsequently proved independently by Ko--Lee--Oh \cite{KLOmax} and Beltran--Guo--Hickman--Seeger \cite{MR4908993}.

In higher dimensions, Ko--Lee--Oh proved local smoothing estimates of sharp order in the high-$p$ range and obtained the first nontrivial maximal bound for $n\geq4$, namely $p>2(n-1)$ \cite[Theorems~1.3--1.4]{KLOsm}. Beltran--Hickman identified an additional necessary condition and formulated a revised local smoothing conjecture \cite{BH25}. Gan-Maldague-Oh proved this revised conjecture in all dimensions \cite{gan2025sharplocalsmoothingestimates}. Their result yields the conjectural maximal range $p>n$ for $n=2,3,4$ and the range $p>2(n-2)$ for $n\geq5$.

In dimensions $n \geq 5$,
the exponent $2(n-2)$ is the natural barrier for the local smoothing approach. Our main theorem improves the barrier by a positive amount.
Whereas the arguments in \cite{KLOsm, gan2025sharplocalsmoothingestimates} deduce maximal estimates from local smoothing on the full Euclidean space $\R^{n+1}$, our argument uses a local smoothing estimate on an $n$-dimensional Katz--Tao set in $\R^{n+1}$.
\medskip

To state the estimates on fractal sets, we introduce the following notation.
For $\delta>0$, let $\mathcal D_\delta^{(m)}$ denote the partition of $\mathbb R^m$ into half-open $\delta$-cubes, and write $|F|_\delta$ for the number of such cubes meeting $F$.

\begin{definition}
 Let $\alpha,C,\delta > 0$ and $m\in\mathbb Z_{>0}$. A measurable set $E$ in $\R^m$ is called a Katz--Tao $(\delta,\alpha,C)$-set if
 \begin{equation}
 \sup_{Q_r\in\mathcal{D}_r^{(m)}} |E\cap Q_r|_\delta \leq C\cdot (r/\delta)^{\alpha}
 \end{equation}
 for every $r\geq \delta$.
\end{definition}

The notion of Katz--Tao $(\delta,\alpha,C)$-sets first appears in \cite{KT01}. By the Littlewood-Paley decomposition and discretization, Theorem \ref{0524.thm51} follows from the following theorem.

\begin{theorem}\label{0525.thm13}
Let $n \geq 5$ and $k\in\mathbb{Z}_{\geq0}$. Let $E$ be a Katz--Tao $(2^{-k},n,C)$-set in $\mathbb{R}^{n}\times [1,2]$. The following holds for some $\epsilon_0=\epsilon_0(n,C)>0$: for every $p>2n-4-\epsilon_0$, there exists $\epsilon=\epsilon(p,n,C)>0$ such that
 \begin{equation}
 \|A_tf\|_{L^p_{x,t}(E)} \lesssim_{p,n,C} 2^{-k\left(\frac1p+\epsilon\right)} \|f\|_{L^p(\R^n)}
 \end{equation}
 for any function $f$ whose Fourier transform is supported on $B(0,2^{k+1})\setminus B(0,2^{k-1})$.
\end{theorem}

For $0<h\leq1$, set $\mathbb I_h:=h\mathbb Z \cap [-1,1]$. Define
\begin{equation}
\mathcal G_h
:=
\left\{
\theta_{s,h}:=[-1,1]\cap[s-h,s+h]:
s\in\mathbb I_h
\right\},
\end{equation}
which covers $[-1,1]$ with overlap at most $3$. For $\theta=\theta_{s,h}$, write $s_\theta:=s$ and call it the center of $\theta$. 
Define $\Gamma(s):=(\gamma(s),1)\in\mathbb R^{n+1}$. For $R\ge1$ and an interval $\theta\in\mathcal G_{R^{-1/n}}$, define
\begin{equation}
\Pi_{R}^{\varsigma}(\theta)
:=
\left\{
\zeta\in\mathbb R^{n+1}:
\begin{array}{ll}
|\langle \zeta,\Gamma^{(k)}(s_\theta)\rangle|
\lesssim R^{-(n-k)/n},
& 0\le k\le n-1,\\
\varsigma\langle \zeta,\Gamma^{(n)}(s_\theta)\rangle \sim 1
\end{array}
\right\},\qquad \varsigma\in\{+,-\}.
\end{equation}
The set $\Pi_{R}^{\varsigma}(\theta)$ is a frequency plate with widths $1,R^{-1/n},\ldots,R^{-1}$ in the frame dual to $(\Gamma(s_\theta),\dots,\Gamma^{(n)}(s_\theta))$.
\medskip

The main ingredient in the proof of Theorem \ref{0525.thm13} is the following fractal $\ell^{p/2}$-estimate. Compared with \cite[Theorem~5.3]{gan2025sharplocalsmoothingestimates}, it gains a fixed power of $R$ when integration is restricted from an $R$-ball $B_R^{(n+1)}$ to an $n$-dimensional Katz–Tao set $E$.

\begin{theorem}\label{0524.thm54}
Let $R\geq1$. Fix $\varsigma \in \{+,-\}$. Let $n \geq 4$ and $4<p<n^2+n-2$. Let $E$ be a Katz--Tao $(1,n,C)$-set in $\R^{n+1}$. Then there exists $c=c(p,n)>0$ such that 
\begin{equation}
\left\|
\sum_{\theta \in \mathcal G_{R^{-1/n}} }F_\theta
\right\|_{L^p(E\cap B_R^{(n+1)})}
\lesssim_{p,n,C}
(R^{\frac1n})^{\frac12-\frac2p-c} 
\left\|
\left(
\sum_{ \theta \in \mathcal G_{R^{-1/n}} } |F_\theta|^{\frac{p}{2}}
\right)^{\frac2p}
\right\|_{L^p(W_{B_R^{(n+1)}} )}
\end{equation}
for any family $(F_{\theta})_{\theta\in\mathcal G_{R^{-1/n}}}$ of functions satisfying $
\operatorname{supp}\widehat{F_\theta}\subset \Pi_R^\varsigma(\theta)$. 
\end{theorem}
The weight $W_{B_R^{(n+1)}}$ is as defined below, to be $\sim1$ on $B_R^{(n+1)}$ and decaying outside of it.

\subsection{Structure of the paper}
The proof has the following structure. In Section~\ref{sec:maxred}, we deduce Theorem~\ref{0524.thm51} from the frequency-localized estimate in Theorem~\ref{0525.thm13}. To prove Theorem~\ref{0525.thm13}, we first write
\begin{equation}
 A_tf= ( \hat{f}(\xi) \widehat{ d\sigma_{\gamma} }(t\xi) )^{\vee}
\end{equation}
where $d\sigma_{\gamma}$ is a weighted curve measure on $\gamma$. Gan-Maldague-Oh \cite{gan2025sharplocalsmoothingestimates} introduced a way of analyzing $\widehat{d\sigma}_{\gamma}$. In Section~\ref{sec:wave-envelope-fractal}, we follow their argument, reduce Theorem~\ref{0525.thm13} to the plate estimate in Theorem~\ref{0524.thm53}, and derive the latter from Theorem~\ref{0524.thm54}.

In Section~\ref{sec:prelim}, we localize the curve and change coordinates so that each piece is a small perturbation of the moment curve, and record necessary estimates to prove Theorem~\ref{0524.thm54}. In Section~\ref{sec:multilin-restrict}, we prove the multilinear restriction estimate for functions Fourier-supported on the plates $\{\Pi_R^{\varsigma}(\theta) \}_{\theta}$. In Section~\ref{sec:broad-narrow}, we prove Theorem~\ref{0524.thm54}, combining the broad-narrow analysis of \cite{MR2860188} with the multilinear restriction estimate.

\subsection{Notation}
The $\varepsilon$-neighborhood of a set $X$ is defined as $N_\varepsilon(X):=\{x:\operatorname{dist}(x,X)\leq\varepsilon\}$. For a ball or a box $U$, denote by $CU$ the concentric dilation of $U$ by a factor $C$. We write $A\lesssim B$ if $A=O(B)$, and write $A\lesssim_{\vec\alpha}B$ if the implicit constant depends on the parameters $\vec\alpha$. If $A\lesssim B$ and $B\lesssim A$, then we write $A\sim B$.

For $m\in\mathbb Z_{>0}$ and $\delta>0$, let $\mathcal D_\delta^{(m)}:=\left\{k+ [0,\delta)^m:k\in(\delta\mathbb Z)^m\right\}$. These half-open cubes partition $\R^m$, have volume $\delta^m$, and have diameter comparable to $\delta$. For $F\subset\R^m$, write
\begin{equation}
 \mathcal D_\delta^{(m)}(F):=\{Q\in\mathcal D_\delta^{(m)}:Q\cap F\neq\varnothing\},
 \qquad |F|_\delta:=\#\mathcal D_\delta^{(m)}(F).
\end{equation}
We write $\mathcal D_\delta=\mathcal D_\delta^{(m)}$ if the ambient dimension is unambiguous.

For $m\in\mathbb Z_{>0}$, define the weight $W^{(m)}:\mathbb{R}^m \rightarrow \mathbb{R}$ by
\begin{equation}
 W^{(m)}(x)=(1+|x|)^{-100m}.
\end{equation}

Let $B_1^{(m)}$ be the Euclidean unit ball and let $Q_1^{(m)}=[-1,1]^m$. For every ball $U$, fix an invertible affine map $T_U$ satisfying $T_U(B_1^{(m)})=U$; for every affine box $U$, fix one satisfying $T_U(Q_1^{(m)})=U$. Define
\begin{equation}
W_U^{(m)}(x):=W^{(m)}(T_U^{-1}(x)),
\qquad
w_U^{(m)}:=\frac{W_U^{(m)}}{\int W_U^{(m)}}.
\end{equation}
We usually suppress the superscript $m$.

For a nonnegative weight $w$, write
\begin{equation}
 \|F\|_{L^p(w)}:=\left(\int|F(x)|^p w(x)\,dx\right)^{1/p}.
\end{equation}
For a set $U$ of finite positive measure, write $\|F\|_{L^p_\#(U)}:=|U|^{-1/p}\|F\|_{L^p(U)}$.

For functions $f: \mathbb{R}^n \rightarrow \mathbb{C}$ with variables $x \in \mathbb{R}^n$, define $\mathcal{F}_x(f)(\xi)$ to be the Fourier transform. For functions $F:\mathbb{R}^n \times \mathbb{R} \rightarrow \mathbb{C}$ with variables $(x,t) \in \mathbb{R}^n \times \mathbb{R}$, define $\mathcal{F}_{x,t}F(\xi,\tau)$ to be the full Fourier transform. The corresponding inverse transforms are denoted by $\mathcal F_\xi^{-1}$ and $\mathcal F_{\xi,\tau}^{-1}$, respectively. We simply denote the transforms by $\hat{\cdot}$ and $(\cdot)^\vee$ when the transformed variables are clear.

\subsection{Tool and computational resource disclosure}
During the development of this work, the authors used ChatGPT (OpenAI, GPT-5.5, 5.6) to explore possible proof strategies. All mathematical statements, proofs, and references appearing in the final manuscript were independently checked and verified by the authors, who take full responsibility for their correctness.

\subsection{Acknowledgements}

Changkeun Oh was supported by the POSCO Science Fellowship of POSCO TJ Park Foundation, and the National
Research Foundation of Korea (NRF) grant funded by the Korea government (MSIT) RS-2024-00341891.

\section{Theorem~\ref{0525.thm13} implies Theorem~\ref{0524.thm51} }\label{sec:maxred}

We reduce the maximal bound to a sum over frequency-localized pieces using Littlewood--Paley decomposition. We linearize the maximal operator using Bernstein's inequality and apply Theorem~\ref{0525.thm13}.

\subsection{Reduction to the local maximal estimate}

For an interval $I\subset(0,\infty)$, write $\mathscr M_I f:=\sup_{t\in I}|A_t f|$, and for $j\in\mathbb Z$ write $\mathscr M_j:=\mathscr M_{[2^j,2^{j+1}]}$. After decreasing the positive constant $\epsilon_0$ in Theorems \ref{0524.thm51} and \ref{0525.thm13}, if necessary, we may assume $0<\epsilon_0\leq1$. This only weakens both assertions, and it ensures that $p>2n-5\geq5$, so in particular $p\geq2$.

\begin{lemma}[Local-to-global reduction]\label{lem:local-to-global-maximal}
Let $2\leq p<\infty$. Suppose that there are constants $\sigma_{\rm dec}>0$ and $C_{\mathrm{loc}}<\infty$ such that
\begin{equation}\label{eq:local-annular-maximal}
\|\mathscr M_{[1,2]}g\|_{L^p(\mathbb R^n)}\leq C_{\mathrm{loc}}2^{-\sigma_{\rm dec}m}\|g\|_{L^p(\mathbb R^n)}
\end{equation}
whenever $m\geq1$ and $g\in\mathcal{S}(\R^n)$ satisfies $\operatorname{supp}\widehat g\subset B(0,2^{m+1})\setminus B(0,2^{m-1})$. Then
\begin{equation}\label{eq:global-from-local}
\|\mathscr M f\|_{L^p(\mathbb R^n)}\lesssim_{p,\sigma_{\rm dec},C_{\mathrm{loc}}}\|f\|_{L^p(\mathbb R^n)}.
\end{equation}
\end{lemma}

\begin{proof}
Choose a radial function $\varphi_{\rm LP}\in C_c^\infty(\mathbb R^n)$ such that $\varphi_{\rm LP}\equiv1$ on $B(0,1)$ and $\operatorname{supp}\varphi_{\rm LP}\subset B(0,2)$. For $\ell\in\mathbb Z$, define the Fourier projection operators $\Delta_\ell$ and $\Delta_{\leq\ell}$ by
\begin{align}
\widehat{\Delta_\ell f}(\xi)&:=\big(\varphi_{\rm LP}(2^{-\ell}\xi)-\varphi_{\rm LP}(2^{-(\ell-1)}\xi)\big)\widehat f(\xi),\label{eq:homogeneous-delta}\\
\widehat{\Delta_{\leq\ell} f}(\xi)&:=\varphi_{\rm LP}(2^{-\ell}\xi)\widehat f(\xi).\label{eq:homogeneous-low-pass}
\end{align}
Then $\operatorname{supp}\widehat{\Delta_\ell f}\subset\{2^{\ell-1}\leq|\xi|\leq2^{\ell+1}\}$, and for every $j\in\mathbb Z$ one has
\begin{equation}\label{eq:relative-frequency-decomposition}
f=\Delta_{\leq -j}f+\sum_{m=1}^\infty\Delta_{m-j}f.
\end{equation}

Let $\rho=\check\varphi_{\rm LP}$. For $t=2^ju$ with $1\leq u\leq2$, Fubini's theorem implies
\begin{equation}\label{eq:low-relative-kernel}
A_{2^ju}\Delta_{\leq -j}f(x)=f*K_{j,u}(x),\qquad K_{j,u}(y)=2^{-jn}\int\rho(2^{-j}y-u\gamma(s))\psi(s)\,ds.
\end{equation}
Since $\gamma$ is bounded on $\operatorname{supp}\psi$, for every $N>0$ one has, uniformly in $j\in\mathbb Z$ and $u\in[1,2]$,
\begin{equation}\label{eq:low-relative-kernel-decay}
|K_{j,u}(y)|\lesssim_N2^{-jn}(1+2^{-j}|y|)^{-N}.
\end{equation}
Taking $N>n$ and decomposing this radial majorant into dyadic annuli yields
\begin{equation}\label{eq:low-relative-HL}
\sup_{j\in\mathbb Z}\mathscr M_j(\Delta_{\leq -j}f)(x)\lesssim\mathscr M_{\mathrm{HL}}f(x),
\end{equation}
where $\mathscr M_{\mathrm{HL}}$ denotes the Hardy--Littlewood maximal operator.

For each $m\geq1$, define
\begin{equation}
G_m(x):=\sup_{j\in\mathbb Z}\mathscr M_j(\Delta_{m-j}f)(x).
\end{equation}
For each $j$, set $g_{j,m}(y):=(\Delta_{m-j}f)(2^jy)$. Its Fourier support is contained in $\{2^{m-1}\leq|\xi|\leq2^{m+1}\}$, and changing variables gives
\begin{equation}
\mathscr M_j(\Delta_{m-j}f)(2^jy)=\mathscr M_{[1,2]}g_{j,m}(y).
\end{equation}
Consequently, \eqref{eq:local-annular-maximal} and a change of variables imply
\begin{equation}\label{eq:scaled-local-annular}
\|\mathscr M_j(\Delta_{m-j}f)\|_{L^p}\lesssim2^{-\sigma_{\rm dec}m}\|\Delta_{m-j}f\|_{L^p},
\end{equation}
uniformly in $j$. Since $\sup_j|h_j|^p\leq\sum_j|h_j|^p$, it follows that
\begin{equation}\label{eq:Gm-first-bound}
\|G_m\|_{L^p}^p\leq\sum_{j\in\mathbb Z}\|\mathscr M_j(\Delta_{m-j}f)\|_{L^p}^p\lesssim2^{-\sigma_{\rm dec}mp}\sum_{\ell\in\mathbb Z}\|\Delta_\ell f\|_{L^p}^p.
\end{equation}
Since $p\geq2$,
\begin{equation}
\sum_{\ell\in\mathbb Z}|\Delta_\ell f|^p\leq\left(\sum_{\ell\in\mathbb Z}|\Delta_\ell f|^2\right)^{p/2}.
\end{equation}
The homogeneous Littlewood--Paley square-function inequality therefore gives
\begin{equation}\label{eq:annular-lp-sum}
\sum_{\ell\in\mathbb Z}\|\Delta_\ell f\|_{L^p}^p\leq\left\|\left(\sum_{\ell\in\mathbb Z}|\Delta_\ell f|^2\right)^{1/2}\right\|_{L^p}^p\lesssim_p\|f\|_{L^p}^p.
\end{equation}
Combining \eqref{eq:Gm-first-bound} and \eqref{eq:annular-lp-sum}, we obtain
\begin{equation}\label{eq:Gm-decay}
\|G_m\|_{L^p}\lesssim2^{-\sigma_{\rm dec}m}\|f\|_{L^p}.
\end{equation}

Apply $\mathscr M_j$ to \eqref{eq:relative-frequency-decomposition}, use sublinearity, and then take the supremum over $j$. By \eqref{eq:low-relative-HL} and \eqref{eq:Gm-decay}, we have
\begin{align}
\mathscr M f(x)&\leq\sup_{j\in\mathbb Z}\mathscr M_j(\Delta_{\leq -j}f)(x)+\sum_{m=1}^\infty G_m(x),\notag\\
\|\mathscr M f\|_{L^p}&\lesssim\|\mathscr M_{\mathrm{HL}}f\|_{L^p}+\sum_{m=1}^\infty2^{-\sigma_{\rm dec}m}\|f\|_{L^p}\lesssim\|f\|_{L^p}.
\end{align}
This proves \eqref{eq:global-from-local} for Schwartz functions, which extends to all $f\in L^p(\mathbb R^n)$ by density and sublinearity.
\end{proof}

By Lemma~\ref{lem:local-to-global-maximal}, it remains to prove \eqref{eq:local-annular-maximal}. Fix $k\geq1$ and put $R=2^k$. Let $f_k\in\mathcal S(\mathbb R^n)$ satisfy $\operatorname{supp}\widehat{f_k}\subset\{\xi:R/2\leq|\xi|\leq2R\}$. It suffices to prove that
\begin{equation}\label{eq:local-frequency-decay-section6}
\|\mathscr M_{[1,2]}f_k\|_{L^p(\mathbb R^n)}\lesssim R^{-\epsilon}\|f_k\|_{L^p(\mathbb R^n)},
\end{equation}
which shows \eqref{eq:local-annular-maximal} with $m=k$ and $\sigma_{\rm dec}=\epsilon$.

Choose a Schwartz function $\chi \in \mathcal{S}(\R)$ such that $|\chi(t)| \ge 1$ for $t \in [1, 2]$ and $\operatorname{supp}\hat\chi\subseteq[-C_\chi,C_\chi]$ for some constant $C_{\chi}>0$. Define
\begin{equation}
 \mathcal H_k(x,t) = A_t f_k(x) \chi(t).
\end{equation}
Since $|\chi(t)| \ge 1$ on the interval $[1,2]$, we have the following pointwise bound:
\begin{equation}
 \mathscr M_{[1,2]}f_k(x) = \sup_{1 \le t \le 2} |A_t f_k(x)| \le \sup_{1 \le t \le 2} |A_t f_k(x) \chi(t)| \le \sup_{t \in \R} |\mathcal H_k(x,t)|.
\end{equation}

\subsection{Linearizing the maximal operator}
Note that
\begin{align}
 \widehat{\mathcal H_k}(\xi, \tau) &= \int_{\R^n} \int_{\R} \left( \int f_k(x - t\gamma(s)) \psi(s)\,ds \right) \chi(t) e^{-i(x \cdot \xi + t \tau)} \,dt \,dx \\
 &= \widehat{f_k}(\xi) \int \psi(s)\widehat{\chi}(\tau + \gamma(s) \cdot \xi) \,ds.
\end{align}
If $\widehat{\mathcal H_k}(\xi,\tau)\neq0$, then $\widehat{f_k}(\xi)\neq0$ and $\widehat{\chi}(\tau+\gamma(s)\cdot\xi)\neq0$ for some $s\in \operatorname{supp}\psi$. Since $\operatorname{supp}\hat\chi\subseteq[-C_\chi,C_\chi]\subseteq[-C_\chi R,C_\chi R]$, it follows that
\begin{equation}\label{eq:supp-Hk}
\operatorname{supp}\widehat{\mathcal H_k}\subseteq\left\{(\xi,\tau):|\xi|\leq 2R,\ |\tau| \le C_{\chi} + C_\gamma|\xi| \le (C_{\chi} + 2C_\gamma) R\right\}
\end{equation}
whenever $R\geq1$, where $C_\gamma:=\sup_{s\in\operatorname{supp}\psi}|\gamma(s)|$.

We use \eqref{eq:supp-Hk} to linearize the maximal operator using the discretization in \cite[Lemma 3.2]{liwu2024maxBR}.

\begin{proposition}\label{prop:linear}
There exists a Katz--Tao $(R^{-1}, n, C)$-set $E \subset \R^{n}\times[1,2]$ such that
\begin{equation}
 \|\mathscr M_{[1,2]}f_k\|_{L^p(\R^n)}^p \lesssim R \int_{\R^{n+1}} |\mathcal H_k(y,t')|^p (\mathbf{1}_E * \Psi_R)(y,t') \,dy \,dt',
\end{equation}
where $\Psi_R(y,t') = R^{n+1} \Psi(Ry, Rt')$ and $\Psi(y, t') = (1 + |y|)^{-N_{\rm dec}}(1 + |t'|)^{-N_{\rm dec}}$ for a large integer $N_{\rm dec}>n+1$ to be chosen below.
\end{proposition}

\begin{proof}
We partition $\R^n$ into a tiling $\mathcal{Q}$ of mutually disjoint cubes $Q$ of side length $R^{-1}$. Also partition the interval $[1,2]$ into subintervals $I_j = [t_j, t_{j+1})$ of length $R^{-1}$. For each cube $Q \in \mathcal{Q}$, let $j(Q)$ be the smallest index that attains the maximum; that is,
\begin{equation}
 \max_{j} \sup_{(x,t) \in Q \times I_j} |\mathcal H_k(x,t)| = \sup_{(x,t) \in Q \times I_{j(Q)}} |\mathcal H_k(x,t)|.
\end{equation}
We group the cubes by defining $F_j = \bigcup_{Q: j(Q) = j} Q$.

For any $x \in Q \subset F_j$, the choice of $j(Q)$ yields
\begin{equation}\label{MtoHLp}
 |\mathscr M_{[1,2]}f_k(x)|^p \le \sup_{t \in [1,2]} |\mathcal H_k(x,t)|^p\le \sup_{(y,t) \in Q \times I_j} |\mathcal H_k(y,t)|^p.
\end{equation}
Let $x_Q$ denote the center of $Q$. It follows from \eqref{eq:supp-Hk} and Bernstein's inequality that
\begin{equation}\label{AppBernstein}
 \sup_{(y,t) \in Q \times I_j} |\mathcal H_k(y,t)|^p \lesssim R^{n+1} \int_{\R^{n+1}} |\mathcal H_k(y,t')|^p \Psi(R(x_Q - y), R(t_j - t')) \,dy \,dt'.
\end{equation}
Since $x \in Q$, we have $R|x - x_Q| \lesssim 1$ which implies $\Psi(R(x_Q - y), R(t_j - t')) \sim \Psi(R(x - y), R(t_j - t'))$. Therefore, \eqref{MtoHLp} and \eqref{AppBernstein} give
\begin{equation}
 |\mathscr M_{[1,2]}f_k(x)|^p \lesssim R^{n+1} \int_{\R^{n+1}} |\mathcal H_k(y,t')|^p \Psi(R(x - y), R(t_j - t')) \,dy \,dt'.
\end{equation}
Integrating over $x \in Q$ and summing over all cubes $Q \subset F_j$, we obtain
\begin{equation}\label{0701.637}
\begin{aligned}
 &\int_{F_j} |\mathscr M_{[1,2]}f_k(x)|^p \,dx\\
 &\qquad\lesssim R \int_{\R^{n+1}} |\mathcal H_k(y,t')|^p \left( \int_{F_j} R^n \Psi_1(R(x-y)) \,dx \right) \Psi_2(R(t_j-t')) \,dy \,dt',
\end{aligned}
\end{equation}
where we split $\Psi(y, t') = \Psi_1(y) \Psi_2(t')$ with $\Psi_1(y)=(1+|y|)^{-N_{\rm dec}}$ and $\Psi_2(t')=(1+|t'|)^{-N_{\rm dec}}$. Since
\begin{equation}
\Psi_2(R(t_j-t'))\lesssim R\int_{I_{j}}\Psi_2(R(t''-t'))\,dt''
\end{equation}
\eqref{0701.637} becomes
\begin{equation}
 \int_{F_j} |\mathscr M_{[1,2]}f_k(x)|^p \,dx \lesssim R \int_{\R^{n+1}} |\mathcal H_k(y,t')|^p (\mathbf{1}_{F_j \times I_j} * \Psi_R)(y,t') \,dy \,dt'.
\end{equation}
Summing over $j$ gives
\begin{equation}\label{eq:appendix-linearization}
 \|\mathscr M_{[1,2]}f_k\|_{L^p(\R^n)}^p \lesssim R \int_{\R^{n+1}} |\mathcal H_k(y,t')|^p (\mathbf{1}_E * \Psi_R)(y,t') \,dy \,dt',
\end{equation}
where the set $E \subset \R^{n}\times[1,2]$ is defined by
\begin{equation}\label{eq:E}
 E = \bigcup_j (F_j \times I_j).
\end{equation}

Each $R^{-1}$-cube $Q$ in $\R^n$ contributes at most one $R^{-1}$-cube $Q\times I_{j(Q)}$ constituting $E$. It thus follows from \eqref{eq:E} that for any $r$-cube $Q_r \subset \R^{n+1}$ of side length $r \ge R^{-1}$,
\begin{equation}
 |E\cap Q_r|_{R^{-1}}\lesssim|\operatorname{proj}_{\R^n}(E\cap Q_r)|_{R^{-1}}\leq|\operatorname{proj}_{\R^n}( Q_r)|_{R^{-1}}\lesssim(r/R^{-1})^n.
\end{equation}
This confirms that $E$ is a Katz--Tao $(R^{-1}, n, C)$-set in $\R^{n+1}$.
\end{proof}

\subsection{Applying Theorem~\ref{0525.thm13}}
Expanding the convolution in \eqref{eq:appendix-linearization} and translating $E$ gives
\begin{equation}
\|\mathscr M_{[1,2]}f_k\|_{L^p}^p
\lesssim
R\int_{\mathbb R^{n+1}}\Psi_R(u,v)
\int_{E+(u,v)}|\mathcal H_k(y,t')|^p\,dy\,dt'\,du\,dv.
\end{equation}
If $|v|\leq1/4$, then $E+(u,v)\subset\mathbb R^n\times[3/4,9/4]$. We may replace the interval $[1,2]$ in Theorem~\ref{0525.thm13} by $[3/4,9/4]$, via a finite partition and fixed rescalings in $t$. It follows that
\begin{equation}
\int_{E+(u,v)}|\mathcal H_k|^p
\lesssim
R^{-(1+p\epsilon)}\|f_k\|_{L^p}^p.
\end{equation}
If $|v|>1/4$, the uniform $L^p$ boundedness of $A_t$ gives the bound
\begin{equation}
\int_{E+(u,v)}|\mathcal H_k|^p\lesssim\|f_k\|_{L^p}^p.
\end{equation}
Choosing $N_{\rm dec}>\max\{n+1,2+p\epsilon\}$, we have
\begin{equation}
\int_{\mathbb R^n}\int_{|v|>1/4}\Psi_R(u,v)\,dv\,du\lesssim R^{1-N_{\rm dec}}.
\end{equation}
Consequently,
\begin{equation}
\|\mathscr M_{[1,2]}f_k\|_{L^p}^p
\lesssim
\bigl(R^{-p\epsilon}+R^{2-N_{\rm dec}}\bigr)\|f_k\|_{L^p}^p
\lesssim
R^{-p\epsilon}\|f_k\|_{L^p}^p.
\end{equation}
Taking the $p$-th root proves \eqref{eq:local-frequency-decay-section6}, and Lemma~\ref{lem:local-to-global-maximal} completes the proof.

\section{Reduction to Theorem~\ref{0524.thm54}}\label{sec:wave-envelope-fractal}

In this section, we deduce Theorem~\ref{0525.thm13} from Theorem~\ref{0524.thm54}.

We first define the multiscale parameters used in the plate reduction. Let $\vec\delta=(\delta_1,\ldots,\delta_{n-1})\in(0,1]^{n-1}$ have dyadic entries, and let $\vec\nu=(\nu_1,\ldots,\nu_{n-1})\in\{-1,0,1\}^{n-1}$. Set
\begin{equation}
\rho_0:=1,
\qquad
\rho_j:=\delta_1^j\delta_2^{j-1}\cdots\delta_j
\quad(1\leq j\leq n-1),
\end{equation}
and $\rho_n:=\delta_1^n\delta_2^{n-1}\cdots\delta_{n-1}^2$. We say that $(\vec\delta,\vec\nu)$ is admissible at scale $R$ if $\rho_n=R^{-1}$ and
\begin{equation}
\nu_i=0
\quad\Longleftrightarrow\quad
\delta_i=1
\qquad(1\leq i\leq n-1).
\end{equation}
We extend the sign vector by setting $\nu_n=\nu_{n+1}:=0$. We also put
\begin{equation}
\delta:=\Delta(\vec\delta):=\prod_{i=1}^{n-1}\delta_i.
\end{equation}

Throughout this section, $(\vec\delta,\vec\nu)$ is an admissible pair at scale $R$.

\subsection{Decomposition of frequencies}

In this subsection, we deduce Theorem~\ref{0525.thm13} from Theorem~\ref{0524.thm53}. We begin with the plate decomposition used in the reduction.

For $s\in(-1,1)$, set
\begin{equation}
v_1(s):=\frac{(\gamma(s),1)}{\sqrt{1+|\gamma(s)|^2}}\in\mathbb R^{n+1}.
\end{equation}
The nondegeneracy of $\gamma$ implies that $v_1(s),v_1'(s),\ldots,v_1^{(n)}(s)$ are linearly independent. Applying the Gram--Schmidt process in this order gives an orthonormal frame $v_1(s),\ldots,v_{n+1}(s)$. We use the reversed frame
\begin{equation}
e_i(s):=v_{n+2-i}(s),
\qquad
1\leq i\leq n+1.
\end{equation}
Note that up to sign, $e_1$ is the dual curve of $v_1$ and $(e_1(s),\ldots,e_{n+1}(s))$ is its Frenet frame.

Define
\begin{equation}
\begin{split}
 P[\vec{\delta},\vec \nu](s):=\left\{ \sum_{i=1}^{n+1} a_i e_i(s): \begin{cases}
 |a_i| \lesssim R \rho_{i-1} \;\; \text{ if } \nu_i=0 \\
 |a_i| \sim R \rho_{i-1},\; a_i>0 \;\; \text{ if } \nu_i=1
 \\
 |a_i| \sim R \rho_{i-1}, \; a_i<0 \;\; \text{ if } \nu_i=-1
 \end{cases} \right\},
\end{split}
\end{equation}
and a normalized plate
\begin{equation}
\begin{split}
 p[\vec{\delta},\vec\nu](s):=R^{-1}P[\vec{\delta},\vec \nu](s)=\left\{ \sum_{i=1}^{n+1} a_i e_i(s): \begin{cases}
 |a_i| \lesssim \rho_{i-1} &\text{ if } \nu_i=0, \\
 |a_i| \sim \rho_{i-1}, \; a_i>0 &\text{ if } \nu_i=1,
 \\
 |a_i| \sim \rho_{i-1}, \; a_i<0 &\text{ if } \nu_i=-1
 \end{cases} \right\}.
\end{split}
\end{equation}
To simplify the notation, we sometimes use $P[\vec{\delta}](s)$ and $p[\vec{\delta}](s)$ for $P[\vec{\delta},\vec\nu](s)$ and $p[\vec{\delta},\vec\nu](s)$, respectively. For $1\leq J\leq n-1$, let
\begin{equation}
\vec\delta_{(J)}
:=
(\delta_1,\ldots,\delta_J,1,\ldots,1),
\qquad
\vec \nu_{(J)}
:=
(\nu_1,\ldots,\nu_J,0,\ldots,0).
\end{equation}
We abbreviate $p[\vec\delta_{(J)},\vec\nu_{(J)}](s)$ by $p[\vec\delta_{(J)}](s)$.

For a rectangular box $P=\prod_{i=1}^{n+1}{[-l_i,l_i]}$, we say $\phi_P$ is adapted if
\begin{enumerate}
 \item $\phi_P$ is supported in a $C$-dilate of $P$.

 \item $|\partial^{\alpha} \phi_P(x_1,\ldots,x_{n+1})| \lesssim_{\alpha} \prod_{i=1}^{n+1} l_i^{-\alpha_i}$ holds for every multi-index $\alpha$.
\end{enumerate}
For a general rectangular box $P$, we apply an isometry to reduce to the above case. In our paper, when we use the notation $\phi_P$, the function is understood to be adapted to $P$.

We show the following fractal local smoothing estimate from Theorem~\ref{0524.thm54}.

\begin{theorem}\label{0524.thm53}
Let $n\ge5$, and $C>0$. Let $R\ge1$, and assume that $E\subset\mathbb R^n\times [R,2R]$ is a Katz--Tao $(1,n,C)$-set. There exists $\epsilon_0=\epsilon_0(n,C)>0$ such that the following holds: for every $p>2n-4-\epsilon_0$, there exists $\epsilon=\epsilon(p,n,C)>0$ such that
\begin{equation}
\Big\|
\sum_{s\in\mathbb I_\delta}
a_s
\mathcal F_{\xi,\tau}^{-1}
\left((\mathcal F_x f)(\xi)\phi_{p[\vec\delta,\vec\nu](s)}\right)
(\cdot-R\mathbf n_s)
\Big\|_{L^p(E)}
\lesssim_{p,n,C}
\delta^{-1}R^{-1-\epsilon}
\|f\|_{L^p(\mathbb R^n)}
\end{equation}
whenever $\operatorname{supp}\widehat f\subset B(0,2)$, $(\vec\delta,\vec\nu)$ is admissible at scale $R$, $|a_s|\lesssim1$, and
$\mathbf n_s\in\{u(\gamma(s),1):u\in[1,2]\}$.
\end{theorem}

\begin{proof}[Proof of Theorem~\ref{0525.thm13} assuming Theorem~\ref{0524.thm53}]
Following \cite[(3.49)--(3.51)]{gan2025sharplocalsmoothingestimates}, it suffices to prove
\begin{equation}\label{0406.03}
\left\|
\sum_{s\in\mathbb I_\delta}
a_s
\mathcal F_{\xi,\tau}^{-1}
\left(
(\mathcal F_xf)(\xi)\phi_{P[\vec\delta,\vec\nu](s)}(\xi,\tau)
\right)
(\cdot-\mathbf n_s)
\right\|_{L^p(E)}
\lesssim
\delta^{-1}R^{-1/p-\epsilon}
\|f\|_{L^p(\mathbb R^n)}.
\end{equation}
For brevity, write $P_s:=P[\vec\delta,\vec\nu](s)$ and $p_s:=R^{-1}P_s$. To apply Theorem~\ref{0524.thm53}, put $\widetilde E:=RE$, $g(x):=f(x/R)$, and $\psi_{p_s}(\zeta,\sigma):=\phi_{P_s}(R\zeta,R\sigma)$. Then $\widetilde E\subset\mathbb R^n\times[R,2R]$ is a Katz--Tao $(1,n,C')$-set, $\psi_{p_s}$ is uniformly adapted to $p_s$, and
\begin{equation}
\|g\|_{L^p(\mathbb R^n)}
=
R^{n/p}\|f\|_{L^p(\mathbb R^n)}.
\end{equation}
Moreover, a direct change of variables $(\xi,\tau)=R(\zeta,\sigma)$ gives
\begin{equation}
\mathcal F_{\xi,\tau}^{-1}
\left(
(\mathcal F_xf)(\xi)\phi_{P_s}(\xi,\tau)
\right)
(Z/R-\mathbf n_s)
=
R
\mathcal F_{\zeta,\sigma}^{-1}
\left(
(\mathcal F_xg)(\zeta)\psi_{p_s}(\zeta,\sigma)
\right)
(Z-R\mathbf n_s).
\end{equation}
Since $\operatorname{supp}\hat g\subset B(0,2)\setminus B(0,1/2)$, Theorem~\ref{0524.thm53} applied on $\widetilde E$ gives
\begin{align}
&\left\|
\sum_{s\in\mathbb I_\delta}
a_s
\mathcal F_{\xi,\tau}^{-1}
\left(
(\mathcal F_xf)(\xi)\phi_{P_s}(\xi,\tau)
\right)
(\cdot-\mathbf n_s)
\right\|_{L^p(E)}
\notag\\
&\qquad=
R^{1-(n+1)/p}
\left\|
\sum_{s\in\mathbb I_\delta}
a_s
\mathcal F_{\zeta,\sigma}^{-1}
\left(
(\mathcal F_xg)(\zeta)\psi_{p_s}(\zeta,\sigma)
\right)
(\cdot-R\mathbf n_s)
\right\|_{L^p(\widetilde E)}
\notag\\
&\qquad\lesssim
\delta^{-1}
R^{1-(n+1)/p}
R^{-1-\epsilon}
R^{n/p}
\|f\|_{L^p(\mathbb R^n)}
=
\delta^{-1}R^{-1/p-\epsilon}
\|f\|_{L^p(\mathbb R^n)}.
\end{align}
This proves \eqref{0406.03}.
\end{proof}

\subsection{Propositions from \cite{gan2025sharplocalsmoothingestimates} }

We use two estimates from \cite{gan2025sharplocalsmoothingestimates}. Proposition~\ref{0526.prop23} descends from one intermediate scale to the next, while Proposition~\ref{0526.prop24} bounds the final $L^4$ square function.
Let $p_m=m^2+m-2$ for any $m \geq 2$.

\begin{proposition}[Proposition 9.5 of \cite{gan2025sharplocalsmoothingestimates}]\label{0526.prop23}
For every $1\le J\le n-2$, $4\le p\le p_{n-J}$, and any
$\varepsilon>0$, we have
\begin{equation}\label{eq:sec2-gmo-multiscale-step}
\begin{aligned}
&\left\|
\left(
\sum_{s\in \mathbb I_{\delta_1\ldots\delta_J}}
\left|
\sum_{\substack{
s'\in \mathbb I_{\delta_1\ldots\delta_{J+1}}\\
|s-s'|\le \delta_1\ldots\delta_J
}}
F_{p[\vec{\delta}_{(J+1)}](s')}
\right|^{\frac p2}
\right)^{\frac 2p}
\right\|_{L^p(W_{B_R^{(n+1)}})}
\\
&\qquad\lesssim
(\delta_{J+1})^{-\left(\frac12-\frac2p+\varepsilon\right)}
\left\|
\left(
\sum_{s'\in \mathbb I_{\delta_1\ldots\delta_{J+1}}}
\left|
F_{p[\vec{\delta}_{(J+1)}](s')}
\right|^{\frac p2}
\right)^{\frac 2p}
\right\|_{L^p(W_{B_R^{(n+1)}})}
\end{aligned}
\end{equation}
for any family of functions
$
F_{p[\vec{\delta}_{(J+1)}](s')}:\mathbb R^{n+1}\to\mathbb C$
with
$\operatorname{supp}\widehat{F}_{p[\vec{\delta}_{(J+1)}](s')}
\subset p[\vec{\delta}_{(J+1)}](s').$
\end{proposition}

\begin{proposition}[Proposition 9.4 of \cite{gan2025sharplocalsmoothingestimates}]
\label{0526.prop24}
Let $\mathcal B_R=B_{\mathbb R^{n+1}}((x_0,t_0),R)$ be a ball of radius $R$ centered at $(x_0,t_0)$, and let $B_R^x:=B_{\mathbb R^n}(x_0,R)$ be a ball of radius $R$ centered at $x_0$. For any $\varepsilon>0$, we have
\begin{equation}\label{eq:sec2-gmo-terminal-square-function}
\left\|
\left(
\sum_{s\in \mathbb I_\delta}
\left|a_s
\mathcal F_{\xi,\tau}^{-1}\left(
 (\mathcal F_xf)(\xi)\phi_{p[\vec\delta](s)}(\xi,\tau)
\right)
\bigl((x,t)-R\mathbf n_s\bigr)\right|^2
\right)^{\frac12}
\right\|_{L^4(W_{\mathcal B_R})}^4
\lesssim_{\varepsilon}
R^{-3+\varepsilon}
\|f\|_{L^4(W_{B_R^x})}^4
\end{equation}
whenever $(\vec\delta,\vec\nu)$ is admissible at scale $R$, $|a_s|\lesssim1$, and
$\mathbf n_s\in\{u(\gamma(s),1):u\in[1,2]\}$.
\end{proposition}

\subsection{Proof of Theorem~\ref{0524.thm53} assuming Theorem~\ref{0524.thm54}}
For fixed coefficients $a_s$ and vectors $\mathbf n_s$ satisfying the hypotheses of Theorem~\ref{0524.thm53}, define
\begin{equation}\label{eq:sec2-translated-operators}
\mathscr T_sf(x,t)
:=
a_s
\mathcal F_{\xi,\tau}^{-1}\left(
(\mathcal F_xf)(\xi)\phi_{p[\vec\delta](s)}(\xi,\tau)
\right)
\bigl((x,t)-R\mathbf n_s\bigr),
\qquad
\mathscr Tf:=\sum_{s\in\mathbb I_\delta}\mathscr T_sf.
\end{equation}
Multiplication by $a_s$ and translation in $(x,t)$ do not change Fourier support. In particular,
$\operatorname{supp}\widehat{\mathscr T_sf}\subset p[\vec\delta](s)$.

We first reduce to a local estimate. Let $\mathcal B_R=B_{\mathbb R^{n+1}}((x_0,t_0),R)$ be an arbitrary ball in $\R^{n+1}$, and let $B_R^x=B_{\mathbb R^n}(x_0,R)$ be its projection onto $x$-variables as in Proposition~\ref{0526.prop24}. It suffices to prove
\begin{equation}\label{eq:sec2-local-weighted-target}
\|\mathscr Tf\|_{L^p(E\cap \mathcal B_R)}
\lesssim
\delta^{-1}R^{-1-c_{\ast\ast}}
\|f\|_{L^p(W_{B_R^x})}
\end{equation}
for some $c_{\ast\ast}=c_{\ast\ast}(p,n)>0$. Indeed, assume \eqref{eq:sec2-local-weighted-target}. Cover the slab $\mathbb R^n\times [R,2R]$ by $O(1)$-overlapping $R$-balls. Since $[R,2R]$ is of length $R$, we only need $O(1)$ layers. Within each of these layers, the $x$-projections of the chosen $R$-balls have bounded overlap. Summing the $p$-th powers of the local estimates \eqref{eq:sec2-local-weighted-target} over these $R$-balls using $\sum_B W_{B_R^x}\lesssim1$ gives Theorem~\ref{0524.thm53}.

Recall that $p_n=n^2+n-2$. We treat separately the ranges 
\begin{equation}
  2n-4-\epsilon_0<p<p_n \;\;\;\; \text{ or  } \;\;\;\; p \geq p_n,
\end{equation}
where $\epsilon_0$ is chosen later.

\begin{remark}[Local uniformity]
\label{rem:ms-gain-local-uniformity}
There exist constants $\eta_n>0$ and $\underline c_n>0$ such that the gain in Theorem~\ref{0524.thm54} may be chosen so that $c(p,n)\geq \underline c_n$ whenever $|p-(2n-4)|<\eta_n$. This will be verified at the end of Section~\ref{sec:broad-narrow}.
\end{remark}

\subsubsection{Case 1. $2n-4- \epsilon_0<p<p_n$}
We begin with the following lemma, where we distinguish $\delta_1\le R^{-\alpha_0}$ from $\delta_1> R^{-\alpha_0}$, where $\alpha_0:=1/(2n)$.
\begin{lemma}
\label{lem:sec2-common-descent}
Let $4<p<p_{n}$. For $1\leq j\leq n-1$ set
\begin{equation}
r_j:=\min\{p,p_{n+1-j}\},
\qquad
\mathfrak D_p(\vec\delta)
:=
\prod_{j=1}^{n-1}\delta_j^{2+\frac{r_j}{2}}.
\end{equation}
If $\delta_1\leq R^{-\alpha_0}$, then set $c_1=c(p,n)>0$ from Theorem~\ref{0524.thm54}. If $\delta_1>R^{-\alpha_0}$, set $c_1=0$. Then, for every $\varepsilon>0$,
\begin{equation}
\label{eq:sec2-common-descent-square-1}
\begin{aligned}
\|\mathscr Tf\|_{L^p(E\cap\mathcal B_R)}^p
\lesssim_{\varepsilon,p,n}
R^{-p+1+\varepsilon}
\delta_1^{c_1p}
\delta^{-p}
\mathfrak D_p(\vec\delta)
\|f\|_\infty^{p-4}
\|f\|_{L^4(W_{B_R^x})}^4.
\end{aligned}
\end{equation}
\end{lemma}

\begin{proof}
For $1\leq j\leq n-1$, define
\begin{equation}
\mathscr T_{j,s_j}f
:=
\sum_{\substack{s\in\mathbb I_\delta\\
|s-s_j|\leq\delta_1\ldots\delta_j}}
\mathscr T_sf,
\qquad
\mathcal E_j(q)
:=
\int
\left(
\sum_{s_j\in\mathbb I_{\delta_1\ldots\delta_j}}
|\mathscr T_{j,s_j}f|^{q/2}
\right)^2
W_{\mathcal B_R}.
\end{equation}
If $\delta_1\leq R^{-\alpha_0}$, then $\delta_1^{-n}\geq R^{n\alpha_0}$. Taking an $O(1)$-overlapping cover of $E\cap\mathcal B_R$ by $\delta_1^{-n}$-balls, applying Theorem~\ref{0524.thm54} at each $\delta_1^{-n}$-ball, and summing $p$-th powers gives
\begin{equation}
\|\mathscr Tf\|_{L^p(E\cap\mathcal B_R)}^p
\lesssim
\delta_1^{-p(\frac12-\frac2p-c(p,n))}
\mathcal E_1(p).
\end{equation}
If $\delta_1>R^{-\alpha_0}$, apply Theorem 5.3 of \cite{gan2025sharplocalsmoothingestimates}. Since
\begin{equation}
\|G\|_{L^p(E\cap B_{\delta_1^{-n}})}
\leq
\|G\|_{L^p(B_{\delta_1^{-n}})},
\end{equation}
taking an $O(1)$-overlapping cover of $E\cap\mathcal B_R$ by $\delta_1^{-n}$-balls and summing $p$-th powers gives
\begin{equation}
\|\mathscr Tf\|_{L^p(E\cap\mathcal B_R)}^p
\lesssim_\varepsilon
\delta_1^{-p(\frac12-\frac2p+\varepsilon)}
\mathcal E_1(p).
\end{equation}
Note that we use $\sum_B W_B\lesssim W_{\mathcal B_R}$ where the sum is over the $\delta_1^{-n}$-balls that cover $\mathcal B_R$.

For $1\leq j\leq n-2$, the pointwise $\ell^{r_{j+1}/2}$-$\ell^\infty$ inequality, Proposition~\ref{0526.prop23} with exponent $r_{j+1}$, and the uniform kernel estimate $\sup_u\|\mathcal F_{\xi,\tau}^{-1}(\phi_{p[\vec\delta](s)})(\cdot,u)\|_{L_x^1}\lesssim R^{-1}$ give
\begin{equation}
\begin{aligned}
\mathcal E_j(r_j)
\lesssim{}&
\left(
R^{-1}\prod_{k=j+1}^{n-1}\delta_k^{-1}
\right)^{r_j-r_{j+1}}
\delta_{j+1}^{-r_{j+1}
(\frac12-\frac2{r_{j+1}}+\varepsilon)}
\|f\|_\infty^{r_j-r_{j+1}}
\mathcal E_{j+1}(r_{j+1}).
\end{aligned}
\end{equation}
Iterating this estimate and using $r_1=p$ and $r_{n-1}=4$, we obtain the factor $\|f\|_\infty^{p-4}\mathcal E_{n-1}(4)$. Moreover,
\begin{equation}
\prod_{j=1}^{n-2}
\left(
R^{-1}\prod_{k=j+1}^{n-1}\delta_k^{-1}
\right)^{r_j-r_{j+1}}
=
R^{4-p}\prod_{k=2}^{n-1}\delta_k^{r_k-p}.
\end{equation}
Note that
\begin{equation}
\delta_1^{-p(\frac12-\frac2p-c_1)}
=
\delta_1^{c_1p}
\delta_1^{-p+2+\frac p2}
\end{equation}
and
\begin{equation}
\delta_k^{r_k-p}
\delta_k^{-\frac{r_k}{2}+2}
=
\delta_k^{-p+2+\frac{r_k}{2}},\qquad k\geq2.
\end{equation}
Therefore, the product of all scale factors equals
\begin{equation}
\delta^{-C\varepsilon}\delta_1^{c_1p}
\delta^{-p}R^{4-p}
\prod_{k=1}^{n-1}\delta_k^{2+\frac{r_k}{2}}
=
\delta^{-C\varepsilon}\delta_1^{c_1p}\delta^{-p}R^{4-p}\mathfrak D_p(\vec\delta)
\end{equation}
for some $C>0$. Since $R^{-1}=\prod_{k=1}^{n-1}\delta_k^{n-k+1}\leq\prod_{k=1}^{n-1}\delta_k^2=\delta^2$, we have $\delta^{-1}\leq R^{1/2}$. After taking $\varepsilon$ sufficiently small, the factor $\delta^{-C\varepsilon}$ is absorbed into the $R^\varepsilon$ loss. The estimate \eqref{eq:sec2-common-descent-square-1} then follows from Proposition~\ref{0526.prop24}.
\end{proof}
To bound $\mathfrak D_p(\vec\delta)$ by a power of $R$, define
\begin{equation}
 \omega_m(p)=\min\left\{\frac1m\left(\frac12+\frac2p\right),\frac2p\right\}
\end{equation}
for $m\in\mathbb{Z}_{>0}$ and $p\geq4$.
\medskip

\subsubsection*{Case 1.1. $\delta_1 \leq R^{-\alpha_0}$.}
For every $1\leq k\leq n-1$,
\begin{equation}
 \frac1p\left(2+\frac{r_k}2\right)\geq(n-k+1)\omega_n(p).
\end{equation}
By Lemma~\ref{lem:sec2-common-descent},
\begin{equation}\label{eq:case1}
 \begin{split}
 \|\mathscr Tf\|_{L^p(E\cap\mathcal B_R)}^p&\lesssim_\epsilon \delta^{-p}R^{-p+1-p\omega_n(p)+\epsilon}\delta_1^{c_1p}\|f\|_{L^\infty}^{p-4}\|f\|_{L^4(W_{B_R^x})}^4\\
 &\lesssim_\epsilon\delta^{-p}R^{-p+1-p\omega_n(p)-c_1p\alpha_0+\epsilon}\|f\|_{L^\infty}^{p-4}\|f\|_{L^4(W_{B_R^x})}^4
 \end{split}
\end{equation}
because
\begin{equation}
 \mathfrak D_p(\vec\delta)^{1/p}=\prod_{k=1}^{n-1}\delta_k^{\frac1p(2+\frac{r_k}{2})}\leq\prod_{k=1}^{n-1}\delta_k^{(n-k+1)\omega_n(p)}=R^{-\omega_n(p)}
\end{equation}
\medskip

\subsubsection*{Case 1.2. $\delta_1 > R^{-\alpha_0}$.}
Since the first intermediate scale $\delta_1$ is too large, we use the estimate without the improvement from Theorem~\ref{0524.thm54}. The estimates for $\delta_{2},\ldots,\delta_{n-1}$ give the exponent $\omega_{n-1}(p)$ in place of $\omega_n(p)$.

To write the remaining factor as a power of $\delta_1$, define
\begin{equation}
 \mu(p,n):=n\omega_{n-1}(p)-\frac12-\frac2p.
\end{equation}
Note that
\begin{equation}
 \frac1p\left(2+\frac {r_k} 2\right)\geq(n-k+1)\omega_{n-1}(p)
\end{equation}
for $k\geq2$. By Lemma~\ref{lem:sec2-common-descent},
\begin{equation}\label{eq:case2}
 \begin{split}
 \|\mathscr Tf\|_{L^p(E\cap\mathcal B_R)}^p&\lesssim_\epsilon \delta^{-p}R^{-p+1-p\omega_{n-1}(p)+p\alpha_0\max\{0,\mu(p,n)\}+\epsilon}\|f\|_{L^\infty}^{p-4}\|f\|_{L^4(W_{B_R^x})}^4
 \end{split}
\end{equation}
because
\begin{equation}
 \begin{split}
 \mathfrak D_p(\vec\delta)^{1/p}=\prod_{k=1}^{n-1}\delta_k^{\frac1p(2+\frac{r_k}{2})}&\leq\delta_1^{\frac12+\frac2p}\prod_{k=2}^{n-1}\delta_k^{(n-k+1)\omega_{n-1}(p)}\\
 &=R^{-\omega_{n-1}(p)}\delta_1^{-\mu(p,n)}\\
 &\leq R^{-\omega_{n-1}(p)}R^{\alpha_0\max\{0,\mu(p,n)\}}
 \end{split}
\end{equation}

For $p\geq 2n-4$, $p\omega_n(p)\geq1$ so that $1-p\omega_n(p)-p\alpha_0c(p,n)<0$. Since
\begin{equation}
p\omega_{n-1}(p)=\min\left\{\frac{p+4}{2(n-1)},2\right\},
\end{equation}
we have $p\omega_{n-1}(p)\geq n/(n-1)$ and $np\omega_{n-1}(p)-p/2-2\leq p\omega_{n-1}(p)$. Hence
\begin{equation}\label{eq:large-delta1-exponent}
1-p\omega_{n-1}(p)+\alpha_0\max\left\{0,np\omega_{n-1}(p)-\frac p2-2\right\}
\leq1-(1-\alpha_0)p\omega_{n-1}(p)\leq-\frac1{2(n-1)}.
\end{equation}
For $p$ sufficiently close to $2n-4$, Remark~\ref{rem:ms-gain-local-uniformity} gives $c(p,n)\geq\underline c_n>0$. Since $1-p\omega_n(p)-p\alpha_0\underline c_n$ and the first term of \eqref{eq:large-delta1-exponent} is continuous in $p$, there exists $\epsilon_0=\epsilon_0(n)>0$ such that both terms remain negative whenever $2n-4-\epsilon_0<p<p_n$. Consequently, for every such $p$, there exists $c_\ast=c_\ast(p,n)>0$ such that
\begin{equation}\label{eq:gain-exponent}
\max\left\{1-p\omega_n(p)-p\alpha_0c(p,n),1-p\omega_{n-1}(p)+\alpha_0\max\left\{0,np\omega_{n-1}(p)-\frac p2-2\right\}\right\}<-pc_\ast(p,n).
\end{equation}
Then, choosing $\epsilon\in(0,pc_\ast(p,n)/2)$ for \eqref{eq:case1} and \eqref{eq:case2} gives
\begin{equation}\label{eq:p-to-infty-4}
 \|\mathscr Tf\|_{L^p(E\cap\mathcal B_R)}^p\lesssim \delta^{-p}R^{-p-pc_\ast(p,n)/2}\|f\|_{L^\infty}^{p-4}\|f\|_{L^4(W_{B_R^x})}^4.
\end{equation}

To pass from \eqref{eq:p-to-infty-4} to an $L^p$ bound, we use the standard localization from \cite[(9.9)--(9.11) and (10.2)]{gan2025sharplocalsmoothingestimates}. Put $A:=\|f\|_{L^p(W_{B_R^x})}$ and assume $A>0$. Let $x_0$ be the center of $B_R^x$, and choose $\chi_{\rm loc}\in C_c^\infty(\mathbb R^n)$ such that
\begin{equation}
\chi_{\rm loc}=1\quad\text{on }B(x_0,R^2),
\qquad
\operatorname{supp}\chi_{\rm loc}\subset B(x_0,2R^2).
\end{equation}
Let $P_{\rm freq}$ be a fixed smooth Fourier multiplier operator whose multiplier equals $1$ on the projection onto the $\xi$-variables of every multiplier support occurring in $\mathscr T$, and set $g:=P_{\rm freq}(\chi_{\rm loc}f)$. Then $\mathscr Tg=\mathscr T(\chi_{\rm loc}f)$.

For $s\in\mathbb I_\delta$, let
$K_s:=\mathcal F_{\xi,\tau}^{-1}(\phi_{p[\vec\delta](s)})$.
Then
$\mathscr T_sh(x,t)=a_s\int_{\mathbb R^n}h(y)K_s((x-y,t)-R\mathbf n_s)\,dy$. Repeated integration by parts gives, for every $N\geq1$,
\begin{equation}\label{eq:Tf-far}
\|\mathscr T((1-\chi_{\rm loc})f)\|_{L^p(E\cap\mathcal B_R)}
\lesssim_N
R^{-N}A
\end{equation}
Moreover, by Young's convolution inequality, $\|g\|_{L^p(W_{B_R^x})}\lesssim A$. 

Fix a sufficiently large constant $C_{\rm amp}>1$ such that $\|g\|_\infty\leq R^{C_{\rm amp}}A$. For $\Lambda=\{2^j:R^{-N}\le 2^j\le R^{C_{\rm amp}}\}$, choose $N$ sufficiently large and write
\begin{equation}
g=\sum_{\lambda\in\Lambda}g_\lambda+r,
\qquad
g_\lambda
:=
g\mathbf 1_{\{\lambda A<|g|\leq2\lambda A\}},
\end{equation}
where $R^{-N}\leq\lambda\leq R^{C_{\rm amp}}$, $\#\Lambda\lesssim\log R$, and $\|r\|_\infty\lesssim R^{-N}A$. For every $\lambda\in\Lambda$,
\begin{equation}
\mathscr Tg_\lambda=\mathscr T(P_{\rm freq}g_\lambda),
\qquad
\|P_{\rm freq}g_\lambda\|_\infty\lesssim\lambda A,
\qquad
\|P_{\rm freq}g_\lambda\|_{L^4(W_{B_R^x})}
\lesssim
\|g_\lambda\|_{L^4(W_{B_R^x})},
\end{equation}
so that \eqref{eq:p-to-infty-4} implies
\begin{equation}\label{eq:g-lambda}
 \begin{aligned}
 \|\mathscr Tg_\lambda\|_{L^p(E\cap\mathcal B_R)}^p&=\|\mathscr T(P_{\rm freq}g_\lambda)\|_{L^p(E\cap\mathcal B_R)}^p \\&\lesssim\delta^{-p}R^{-p-pc_\ast(p,n)/2}\|P_{\rm freq}g_\lambda\|_\infty^{p-4}\|P_{\rm freq}g_\lambda\|_{L^4(W_{B_R^x})}^4\\
 &\lesssim \delta^{-p}R^{-p-pc_\ast(p,n)/2}(\lambda A)^{p-4}\|g_\lambda\|_{L^4(W_{B_R^x})}^4\\
 &\sim \delta^{-p}R^{-p-pc_\ast(p,n)/2}\|g_\lambda\|_{L^p(W_{B_R^x})}^p.
 \end{aligned}
\end{equation}
Also, $\|P_{\rm freq}r\|_\infty\lesssim R^{-N}A$. By the uniform kernel estimate $\|\mathscr T_s\|_{L^\infty(\mathbb R^n)\to L^\infty(\mathbb R^{n+1})}\lesssim R^{-1}$, we have
\begin{equation}
\|\mathscr Tr\|_{L^p(E\cap\mathcal B_R)}=\|\mathscr T(P_{\rm freq}r)\|_{L^p(E\cap\mathcal B_R)}
\lesssim
\delta^{-1}R^{n/p-1}R^{-N}A\lesssim
\delta^{-1}R^{-2-c_\ast/2}A
\end{equation}
for $N>2n/p+c_\ast/2+1$. Therefore, by \eqref{eq:g-lambda} and H\"older's inequality in $\lambda$,
\begin{equation}\label{eq:Tf-near}
\begin{aligned}
\|\mathscr T(\chi_{\rm loc}f)\|_{L^p(E\cap\mathcal B_R)}=\|\mathscr Tg\|_{L^p(E\cap\mathcal B_R)}
&\lesssim
\delta^{-1}R^{-1-c_\ast/2}
\sum_{\lambda\in\Lambda}
\|g_\lambda\|_{L^p(W_{B_R^x})}
+
\delta^{-1}R^{-2-c_\ast/2}A \\
&\lesssim
\delta^{-1}R^{-1-c_\ast/2}
(\log R)^{1-1/p}A \\
&\lesssim_\varepsilon
\delta^{-1}R^{-1-c_\ast/2+\varepsilon}
\|f\|_{L^p(W_{B_R^x})}.
\end{aligned}
\end{equation}
By \eqref{eq:Tf-far} and \eqref{eq:Tf-near}, taking $\varepsilon<c_\ast/4$ proves the local estimate \eqref{eq:sec2-local-weighted-target} for $p\in(2n-4-\epsilon_0,p_n)$.

\subsubsection{Case 2. $p \geq p_n$}
We now interpolate the estimate from Case 1 with the $L^\infty$ bound. Fix $p_{\rm small}\in(2n-4-\epsilon_0,p_n)$. By the small-exponent result just proved, there is $c_\ast=c_\ast(p_{\rm small},n)>0$ such that
\begin{equation}\label{eq:sec2-small-p0-endpoint}
\|\mathscr Tf\|_{L^{p_{\rm small}}(E\cap\mathcal B_R)}
\lesssim_{C}
\delta^{-1}R^{-1-c_\ast/4}
\|f\|_{L^{p_{\rm small}}(W_{B_R^x})}.
\end{equation}
Let $\chi$ be a smooth cutoff in the $\xi$ variable which is identically $1$ on $\cup_{s\in\mathbb I_\delta}\operatorname{proj}_\xi(\operatorname{supp}\phi_{p[\vec\delta](s)})$. Define
\begin{equation}
\mathscr T^{\chi}f:=\mathscr T\mathcal F_x^{-1}(\chi\widehat f).
\end{equation}
Following the $L^\infty$ argument in \cite[proof of Theorem~9.1, especially (9.10)]{gan2025sharplocalsmoothingestimates}, the uniform $L^1$ bound $\sup_u\|\mathcal F_{\xi,\tau}^{-1}(\phi_{p[\vec\delta](s)})(\cdot,u)\|_{L_x^1}\lesssim R^{-1}$, together with $|a_s|\lesssim1$ and translation invariance of the $L^1$-norm, gives
\begin{equation}\label{eq:sec2-interp-infty}
\|\mathscr T^\chi f\|_{L^\infty(E\cap\mathcal B_R)}\lesssim\#\mathbb I_\delta \cdot R^{-1}\|f\|_{L^\infty}\lesssim\delta^{-1}R^{-1}\|f\|_{L^\infty}.
\end{equation}
Since $\mathcal F_x^{-1}(\chi\widehat\cdot)$ is uniformly bounded on $L^{p_{\rm small}}(W_{B_R^x})$, it follows that
\begin{equation}\label{eq:sec2-interp-p0}
\|\mathscr T^\chi f\|_{L^{p_{\rm small}}(E\cap\mathcal B_R)}\lesssim\delta^{-1}R^{-1-c_\ast/4}\|f\|_{L^{p_{\rm small}}(W_{B_R^x})}.
\end{equation}
Interpolating \eqref{eq:sec2-interp-p0} and \eqref{eq:sec2-interp-infty} yields, for every finite $p\geq p_n$,
\begin{equation}
\|\mathscr T^\chi f\|_{L^p(E\cap\mathcal B_R)}\lesssim\delta^{-1}R^{-1-c_\ast p_{\rm small}/4p}\|f\|_{L^p(W_{B_R^x})}.
\end{equation}
Since $\mathscr T^\chi f=\mathscr Tf$, \eqref{eq:sec2-local-weighted-target} follows with $c_{\ast\ast}=c_\ast p_{\rm small}/4p$.

\section{Preliminaries for the proof of Theorem~\ref{0524.thm54}}\label{sec:prelim}
In this section, we normalize the curve to be close to the moment curve, and then record the estimates used in the broad-narrow analysis.

\subsection{Localization of a nondegenerate curve}\label{seq:loc-curve}

Let $\gamma_{\rm mo}(s):=(s,s^2,\ldots,s^n)$. Since the cutoff $\psi$ is compactly supported in $(-1,1)$, we may first localize to finitely many compact subintervals of $(-1,1)$. By the standard normalization for a nondegenerate curve using Taylor expansion on short intervals, as used in \cite[Proposition~6.3]{gan2025sharplocalsmoothingestimates}, we can find a change of variables $s=s_0+\rho t$ and an invertible linear map $\mathsf T$ such that
\begin{equation}
\widetilde\Gamma(t)
:=
\mathsf T(\gamma(s_0+\rho t),1)
=
(\widetilde\gamma(t),1),
\end{equation}
where
\begin{equation}\label{eq:ms-moment-curve-reduction}
\widetilde\gamma(t)
=
\gamma_{\rm mo}(t)+E_{\rm loc}(t),
\qquad
\|E_{\rm loc}\|_{C^{N_0}}
\leq
\varepsilon_{\rm loc}
\end{equation}
on a fixed neighborhood of $[-1,1]$.

Choose $N_0\geq 2n+1$ sufficiently large for the cone and the multilinear estimates below to hold. We then choose $\varepsilon_{\rm loc}>0$ sufficiently small to ensure the perturbation bounds required for applying estimates from \cite{MR3783217}. Both choices are independent of $R$ and of the intervals used in subsequent rescalings.
On the frequency side, the map $\mathsf T^{-T}$ carries frequency plates for $\gamma$ to those for $\widetilde{\gamma}$. Indeed, if $\widetilde\zeta:=\mathsf T^{-T}\zeta$, then
\begin{equation}
\left\langle
\widetilde\zeta,
\widetilde\Gamma^{(k)}(t)
\right\rangle
=
\rho^k
\left\langle
\zeta,
\Gamma_{\rm old}^{(k)}(s_0+\rho t)
\right\rangle
\qquad
(0\leq k\leq n),
\end{equation}
where $\Gamma_{\rm old}(s):=(\gamma(s),1)$ denotes the lifted curve before normalization. By abuse of notation, we continue to write $\gamma$ for $\widetilde \gamma$.
\medskip

Put $\Gamma(s):=(\gamma(s),1)$. Since $\Gamma(s),\Gamma'(s),\ldots,\Gamma^{(n)}(s)$ are linearly independent, let $\Xi(s)$ be the unique smooth vector satisfying
\begin{equation}\label{eq:dual-frequency-curve}
\langle\Xi(s),\Gamma^{(k)}(s)\rangle=0
\quad(0\leq k\leq n-1),
\qquad
\langle\Xi(s),\Gamma^{(n)}(s)\rangle=1.
\end{equation}
Differentiating \eqref{eq:dual-frequency-curve} gives
\begin{equation}\label{eq:ms-dual-anti-triangular}
\langle\Xi^{(j)},\Gamma^{(k)}\rangle=0\quad(j+k<n),\qquad \langle\Xi^{(j)},\Gamma^{(n-j)}\rangle=(-1)^j,
\end{equation}
so $\Xi$ satisfies
\begin{equation}
 \det{ \big( \Xi(s),\Xi'(s),\ldots,\Xi^{(n)}(s) \big)} \neq 0.
\end{equation}

By \eqref{eq:ms-moment-curve-reduction}, after applying some linear transformation, we may write $\Xi(s)=a(s)(\kappa(s),1)$ for some smooth curve $\kappa$ satisfying
\begin{equation}\label{eq:ms-dual-graph-chart}
\|\kappa-\gamma_{\rm mo}\|_{C^{N_0-n}([-1,1])}\lesssim\varepsilon_{\rm loc},
\end{equation}
where $a(s)$ is a smooth function satisfying $|a(s)| \sim 1$. Denote by $\mathfrak K$ the collection of curves $\kappa$, defined with uniform $C^{N_0-n}$ bounds on the fixed neighborhood of $[-1,1]$ specified above and satisfying \eqref{eq:ms-dual-graph-chart}.
\medskip

For fixed constants $0<a_-<1<a_+<\infty$, define the following collection of curves 
\begin{equation}\label{eq:ms-uniform-source-class-definition}
\mathfrak E:=\left\{\eta_{\varsigma,a,\kappa}(t)=\varsigma a\int_0^t(\kappa(u),1)\,du:\ \kappa\in\mathfrak K,\ \varsigma\in\{+,-\},\ a\in[a_-,a_+]\right\}.
\end{equation}
Since $\eta'=\varsigma a(\kappa,1)$ and $\kappa$ is a small $C^n$ perturbation of $\gamma_{\rm mo}$, there are constants $C_N^{\mathfrak E}<\infty$ and $\nu_{\det}>0$ such that, for every integer $0\leq N\leq N_0-n+1$,
\begin{equation}\label{eq:ms-uniform-source-class}
\sup_{\eta\in\mathfrak E}\|\eta\|_{C^N([-1,1])}\leq C_N^{\mathfrak E},\qquad \inf_{\eta\in\mathfrak E}\inf_{t\in[-1,1]}\left|\det\bigl(\eta'(t),\ldots,\eta^{(n+1)}(t)\bigr)\right|\geq\nu_{\det}.
\end{equation}
Let $J_{\rm sc}$ be a fixed compact interval in $(0,\infty)$. For $\kappa\in\mathfrak K$, write
\begin{equation}\label{eq:conic-pieces}
\Sigma_I^\kappa:=\{\lambda(\kappa(s),1):\lambda\in J_{\rm sc},\ s\in I\}
\end{equation}
and simply $\Sigma_I$ when $\kappa$ is fixed.

\begin{lemma}[Rescaling lemma]\label{lem:ms-parent-normalization}
Let $J\subset [-\frac12,\frac12]$ have center $s_J$ and length $0<\rho\leq1$. Define the $n \times n$ matrix $\mathsf M_{s_J}$ and the linear transformation $\mathsf A_J$ by
\begin{equation}\label{eq:ms-AJ-definition}
\mathsf M_{s_J}:=\left[\kappa'(s_J),\frac{\kappa''(s_J)}{2!},\ldots,\frac{\kappa^{(n)}(s_J)}{n!}\right],\qquad
\mathsf A_J(\xi',\xi_{n+1}):=\left(D_\rho\mathsf M_{s_J}^{-1}(\xi'-\xi_{n+1}\kappa(s_J)),\xi_{n+1}\right),
\end{equation}
where $D_\rho:=\operatorname{diag}(\rho^{-1},\ldots,\rho^{-n})$.
\begin{enumerate}
\item We have $\|\mathsf A_J\|\lesssim\rho^{-n}$, $\|\mathsf A_J^{-1}\|\lesssim1$, and $|\det\mathsf A_J|\sim\rho^{-n(n+1)/2}$. Moreover,
\begin{equation}
 \mathsf A_J(\kappa(s_J+\rho v),1)=(\widetilde\kappa_J(v),1)\qquad\text{for some $\widetilde\kappa_J\in\mathfrak K$},
\end{equation}
with $\mathsf A_J(\kappa^{(j)}(s_J+\rho v),0)=\rho^{-j}(\widetilde\kappa_J^{(j)}(v),0)$, $1 \leq j \leq n$.

\item Let $r\geq1$. If $I\subset J$, $\operatorname{supp}\widehat G\subset N_{r^{-1}}(\Sigma_I^\kappa)$, and $\widetilde G(y):=G(\mathsf A_J^Ty)$, then
\begin{equation}\label{eq:ms-AJ-support}
\operatorname{supp}\widehat{\widetilde G}\subset N_{1/R_J}(\Sigma_{\widetilde I}^{\widetilde\kappa_J})
\end{equation}
with $\widetilde I:=(I-s_J)/\rho$ and $R_J:=r\rho^n$.
\end{enumerate}
\end{lemma}

\begin{proof}
(1) For $v\in[-1,1]$, Taylor's formula gives $\widetilde\kappa_J(v)=D_\rho\mathsf M_{s_J}^{-1}(\kappa(s_J+\rho v)-\kappa(s_J))=\gamma_{\rm mo}(v)+\mathcal R_J(v)$ with $\|\mathcal R_J\|_{C^n}\lesssim\varepsilon_{\rm loc}\rho$ and uniform bounds for derivatives of order $n+1,\dots,N_0-n$; hence $\widetilde\kappa_J\in\mathfrak K$. The remaining assertions in (1) follow from uniform invertibility of $\mathsf M_{s_J}$ and direct calculation.

(2) The assertion follows from $\widehat{\widetilde G}(\xi)=|\det\mathsf A_J|^{-1}\widehat G(\mathsf A_J^{-1}\xi)$ and $\|\mathsf A_J\|r^{-1}\lesssim(r\rho^n)^{-1}$.
\end{proof}

Since $\widetilde\kappa_J\in\mathfrak K$, the curve $\eta_{\varsigma,a,\widetilde\kappa_J}$ produced by rescaling still belongs to $\mathfrak E$. Thus the subsequent estimates apply uniformly to curves in the class $\mathfrak E$.

\subsection{Some lemmas}

In the proof of Theorem~\ref{0524.thm54} for small $p$, we use the abstract broad-narrow decomposition \cite[Lemma A.2]{MR4908993}. For a dyadic interval $J$, let $\mathcal D(J;\lambda)$ be the dyadic subintervals of $J$ of length $\lambda$. Let $\operatorname{Sep}_k(J;\lambda)$ be the ordered $k$-tuples $(I_1,\dots,I_k)$ in $\mathcal D(J;\lambda)^k$ satisfying $\operatorname{dist}(I_i,I_j)\geq\lambda\quad(i\neq j)$.

\begin{lemma}[Abstract broad-narrow decomposition]
\label{lem:ms-abstract-bn}
Let $(X,\mu_{\rm meas})$ be a measure space, let $k\geq2$, and let $0<\delta<1$ be dyadic. Suppose that $(F_I)_{|I|\geq\delta}$ is a dyadic decomposition of $F\in L^q(X)$, where $1\leq q<\infty$. For every $\beta_{\rm abs}>0$ there are dyadic numbers $\delta_{\rm nar},\delta_{\rm br}$ and a dyadic constant $C_{\beta_{\rm abs},k}\geq1$ such that
\begin{equation}\label{eq:ms-stopping-scales}
 \delta<\delta_{\rm nar}\lesssim_{\beta_{\rm abs},k}\delta,
 \qquad
 \delta<\delta_{\rm br}\leq1,
\end{equation}
and
\begin{align}
\|F\|_{L^q(X)}
\lesssim_{\beta_{\rm abs},k}{}&
\delta^{-\beta_{\rm abs}}
\left(\sum_{|I|=\delta_{\rm nar}}
\|F_I\|_{L^q(X)}^q\right)^{1/q}
\label{eq:ms-abstract-bn}\\
&+\delta^{-\beta_{\rm abs}}
\left(
\sum_{|J|=C_{\beta_{\rm abs},k}\delta_{\rm br}}
\ \sum_{\mathbf I\in
\operatorname{Sep}_k(J;\delta_{\rm br})}
\left\|
\prod_{j=1}^k|F_{I_j}|^{1/k}
\right\|_{L^q(X)}^q
\right)^{1/q}.
\notag
\end{align}
The second sum is understood to be empty if $C_{\beta_{\rm abs},k}\delta_{\rm br}>1$.
\end{lemma}

\begin{proof}
Iterate the variant of Bourgain--Guth decomposition \cite{MR2860188} as in \cite{MR3161099} and use the dyadic identities among the functions $F_I$.
\end{proof}

We also use the following local orthogonality estimate on $r$-cubes.
\begin{lemma}[Local orthogonality]
\label{lem:ms-selected-local-orthogonality}
Let $Q$ be an $r$-cube and let $(H_\alpha)_\alpha$ be a finite family such that
\begin{equation}
 \sum_\alpha\mathbf{1}_{\operatorname{supp}\widehat{H_\alpha}+B(0,cr^{-1})}\lesssim1
\end{equation}
for some fixed $c>0$. Then
\begin{equation}\label{eq:ms-selected-local-orthogonality}
 \int\left|\sum_{\alpha}H_\alpha\right|^2W_{Q}
 \lesssim
 \int\sum_{\alpha}|H_\alpha|^2W_{Q}.
\end{equation}
\end{lemma}

\begin{proof}
Let $x_Q$ be the center of $Q$. We have
\begin{equation}
 W_{Q}(x)\lesssim\sum_{j\ge0}2^{-100(n+1)j}\mathbf 1_{2^jC_0Q}(x)
\end{equation}
for some $C_0=C_0(n)$. For each $j\ge0$, choose a Schwartz function $\psi_{j,Q}$ satisfying $|\psi_{j,Q}|\ge1$ on $2^jC_0Q$, $\operatorname{supp}\widehat{\psi_{j,Q}}\subset B(0,(c/2)2^{-j}r^{-1})$, and
\begin{equation}
 |\psi_{j,Q}(x)|\lesssim_N\left(1+\frac{|x-x_Q|}{2^jr}\right)^{-N}
\end{equation}
for every fixed $N$. Such functions are obtained by translating and dilating one fixed Schwartz function with compact Fourier support. Plancherel's theorem gives
\begin{equation}
 \int_{2^jC_0Q}\left|\sum_\alpha H_\alpha\right|^2
 \leq\left\|\psi_{j,Q}\sum_\alpha H_\alpha\right\|_2^2
 \lesssim\sum_\alpha\|\psi_{j,Q}H_\alpha\|_2^2,
\end{equation}
because $\operatorname{supp}(\widehat{\psi_{j,Q}}*\widehat H_\alpha)\subset\operatorname{supp}\widehat H_\alpha+B(0,cr^{-1})$. Multiplying by $2^{-100(n+1)j}$ and summing in $j$, then taking $N>200(n+1)$, yields
\begin{equation}
 \sum_{j\ge0}2^{-100(n+1)j}|\psi_{j,Q}(x)|^2\lesssim W_{C'Q}(x)
\end{equation}
for a fixed $C'=C'(c,n)$. The preceding estimates imply \eqref{eq:ms-selected-local-orthogonality}.
\end{proof}

\subsection{Cone estimates, rescaling, and interpolation}

We record the sharp cone $\ell^{q/2}$ estimate of Gan-Maldague-Oh. Let us first introduce the necessary definitions.

\begin{definition}
\label{def:ms-admissible-cone-blocks}
Fix a compact interval $J\subset(0,\infty)$. For $\eta\in\mathfrak E$, $0<\rho\leq1$, and $s\in [-\frac12,\frac12]$, define
\begin{equation}
\mathfrak B_\eta(s;\rho)
:=
\left\{
\sum_{j=0}^{n}\lambda_j\eta^{(j+1)}(s):
\lambda_0\in J,\quad
|\lambda_j| \lesssim \rho^j
\quad(1\leq j\leq n)
\right\}.
\end{equation}
Let $s_\theta$ denote the center of $\theta\in\mathcal G_\delta$. A finite family $(K_\theta)_{\theta\in\mathcal G_\delta}$ is called admissible at scale $\delta^{-n}$ if
\begin{equation*}
    \operatorname{supp}\widehat K_\theta
\subset
\mathfrak B_\eta(s_\theta;\delta)
\end{equation*}
for every $\theta$.
\end{definition}

The following estimate applies to the admissible families in Definition~\ref{def:ms-admissible-cone-blocks}.
\begin{proposition}{\cite[Theorem~5.3]{gan2025sharplocalsmoothingestimates}}
\label{cor:ms-family-form-gmo}
Fix $4\leq q\leq n^2+n-2$ and $\varepsilon>0$, and choose $N_0$ sufficiently large depending on $n,q,\varepsilon$. Let $R\geq1$ and $\eta\in\mathfrak E$. If $(K_\theta)_{\theta\in\mathcal G_{R^{-1/n}}}$ is admissible at scale $R$, then
\begin{equation}\label{eq:ms-family-form-gmo}
\left\|\sum_\theta K_\theta\right\|_{L^q}
\lesssim_{\mathfrak E,q,\varepsilon}
R^{\frac1n\left(\frac12-\frac2q\right)+\varepsilon}
\left\|
\left(\sum_\theta|K_\theta|^{q/2}\right)^{2/q}
\right\|_{L^q}.
\end{equation}
\end{proposition}

Rescaling an interval of length $\rho$ gives the following estimate at scale $R \rho^n$.

\begin{lemma}[Rescaled cone $\ell^{q/2}$ estimate]
\label{lem:ms-rescaled-envelope}
Fix $4\leq q\leq n^2+n-2$ and $\varepsilon>0$, and choose $N_0$ as in Proposition~\ref{cor:ms-family-form-gmo}. Let $\eta\in\mathfrak E$ and let $I\subset [-\frac12,\frac12]$ be an interval of length $\rho$.
Let $(H_\theta)_{\theta\in\mathcal G_{R^{-1/n}}}$ be admissible at scale $R$, and
suppose that $H_\theta\equiv0$ whenever $s_\theta\notin I$. If $ R \rho^n \geq1$, then
\begin{equation}\label{eq:ms-rescaled-envelope}
\left\|\sum_\theta H_\theta\right\|_{L^q(B_R^{(n+1)})}
\lesssim_{\varepsilon,q,n,\mathfrak K}
 \big( R \rho^n \big)^{\frac1n\left(\frac12-\frac2q\right)+\varepsilon}
\left\|
\left(\sum_\theta|H_\theta|^{q/2}\right)^{2/q}
\right\|_{L^q(W_{B_R^{(n+1)}})}.
\end{equation}
\end{lemma}

\begin{proof}
Let $\mathsf A_I$ be the linear map from Lemma~\ref{lem:ms-parent-normalization} with $J=I$. Define
$G_\theta(y):=H_\theta(\mathsf A_I^Ty).$
By Lemma~\ref{lem:ms-parent-normalization}(1), the original blocks transform into blocks
at scale $R\rho^n$, centered at $(s_\theta-s_I)/\rho$. Thus Proposition~\ref{cor:ms-family-form-gmo} gives
\begin{equation}
\left\|\sum_\theta G_\theta\right\|_{L^q}
\lesssim
(R\rho^n)^{\frac1n(\frac12-\frac2q)+\varepsilon}
\left\|
\left(\sum_\theta |G_\theta|^{q/2}\right)^{2/q}
\right\|_{L^q}.
\end{equation}
Changing variables $x=\mathsf A_I^Ty$ cancels the common Jacobian factor on both sides. We then obtain
\begin{equation}
\left\|\sum_\theta H_\theta\right\|_{L^q(B_R)}
\lesssim
(R\rho^n)^{\frac1n(\frac12-\frac2q)+\varepsilon}
\left\|
\left(\sum_\theta |H_\theta|^{q/2}\right)^{2/q}
\right\|_{L^q(W_{B_R})}
\end{equation}
by localizing to the $R$-ball $B_R$.
\end{proof}

The following interpolation theorem is proved in \cite[Proposition~6.5]{gan2025sharplocalsmoothingestimates} in the case where $Y=\R^{n+1}$. The same proof works for any measurable set $Y \subset \R^{n+1}$. We do not reproduce the proof here.

\begin{lemma}[Interpolation]
\label{lem:ms-large-p-admissible-interpolation}
Let $\eta\in\mathfrak E$, and let $Y\subset 2B_R\subset\mathbb R^{n+1}$ be measurable. Let $(K_\theta)_{\theta\in\mathcal G_{R^{-1/n}}}$ be admissible at scale $R$. Let
\begin{equation}
 4<q_0<q<q_1\leq n^2+n-2,
 \qquad
 \frac1q=\frac\lambda{q_0}+\frac{1-\lambda}{q_1},
 \qquad 0<\lambda<1.
\end{equation}
Suppose that constants $A_0,A_1\geq1$ satisfy the following: for $j=0,1$, every admissible family $\mathbf G=(G_\vartheta)_{\vartheta\in\mathcal G_{R^{-1/n}}}$ satisfies
\begin{equation}\label{eq:ms-large-p-admissible-endpoints}
 \left\|\sum_\vartheta G_\vartheta\right\|_{L^{q_j}(Y)}
 \leq A_j\left\|\left(\sum_\vartheta |G_\vartheta|^{q_j/2}\right)^{2/q_j}\right\|_{L^{q_j}(\mathbb R^{n+1})}.
\end{equation}
Then, for every $\varepsilon>0$,
\begin{equation}\label{eq:ms-large-p-admissible-interpolation}
 \left\|\sum_\theta K_\theta\right\|_{L^q(Y)}
 \lesssim_{q_0,q,q_1,\varepsilon}
 R^{\varepsilon} A_0^\lambda A_1^{1-\lambda}
 \left\|\left(\sum_\theta |K_\theta|^{q/2}\right)^{2/q}\right\|_{L^q(\mathbb R^{n+1})}.
\end{equation}
The implicit constant is uniform over curves in $\mathfrak E$.
\end{lemma}

\section{Multilinear restriction estimate for a cone}\label{sec:multilin-restrict}

For each localized curve piece, we use the following multilinear restriction estimate. Recall \eqref{eq:conic-pieces}. 
\begin{proposition}[Multilinear restriction estimate]
\label{0602.prop31}
Let $A\geq1$, $K\leq R$, and let $I_1,\ldots,I_{n+1}\subset[-1,1]$ be pairwise separated intervals of length $K^{-1}$, with separation comparable to $K^{-1}$. For every $\varepsilon>0$,
\begin{equation}\label{eq:ms-multilinear-restriction}
 \left\|\prod_{\ell=1}^{n+1}|H_\ell|^{1/(n+1)}
 \right\|_{L^{n+1}_{\#}(B_{R})}
 \lesssim_{\varepsilon,A,K,n,\mathfrak K}
 R^{\varepsilon}
 \prod_{\ell=1}^{n+1}
 \|H_\ell\|_{L^2(w_{B_{R}})}^{1/(n+1)}
\end{equation}
whenever $\operatorname{supp}\widehat H_\ell\subset N_{AR^{-1}}(\Sigma_{I_\ell})$. The estimate is uniform over $\kappa\in\mathfrak K$.
\end{proposition}

For a smooth nondegenerate curve $\kappa\in\mathfrak K$ and an interval $I\subset[-1,1]$, write
\begin{equation}
\Sigma_{\kappa,I}:=\{\rho(\kappa(s),1):\rho\in[1,2],\ s\in I\}\subset\mathbb R^{n+1}.
\end{equation}
The conic pieces \eqref{eq:conic-pieces} in Section~\ref{sec:prelim} use the fixed interval $J_{\rm sc}\Subset(0,\infty)$. By partitioning $J_{\rm sc}$ into finitely many subintervals and dilating them, it suffices to prove Proposition~\ref{0602.prop31} with $[1,2]$.

The proof of Proposition~\ref{0602.prop31} follows \cite[Proposition~9.6]{MR4908993}. We first verify the Brascamp--Lieb dimension condition uniformly on the normalized curve class. Compactness and \cite[Theorem~1.1]{MR3783217} then give an open set of Brascamp--Lieb data on which the Brascamp–Lieb constants are uniformly bounded. We then write each function as an integral of extension operators over nearby surfaces and apply \cite[Theorem~1.3]{MR3783217}.

\subsection{Brascamp--Lieb data}
For $u=(\rho,s)\in[1,2]\times[-1,1]$, the tangent plane to $\Sigma_\kappa(\rho,s)=\rho(\kappa(s),1)$ is
\begin{equation}
V_\kappa(s)=\operatorname{span}\{(\kappa(s),1),(\kappa'(s),0)\}\subset\mathbb R^{n+1}.
\end{equation}
Let $\pi_{\kappa,s}$ denote orthogonal projection onto $V_\kappa(s)$. Equivalently, the corresponding rank-two map is
\begin{equation}
L_{\kappa,\rho,s}x=\bigl(x\cdot(\kappa(s),1),\rho x\cdot(\kappa'(s),0)\bigr).
\end{equation}
For every subspace $V\subset\mathbb R^{n+1}$, $\dim L_{\kappa,\rho,s}V=\dim\pi_{\kappa,s}V$. The Brascamp--Lieb dimension condition is
\begin{equation}\label{eq:BL_condition}
\dim V\leq\frac12\sum_{i=1}^{n+1}\dim\pi_{\kappa,s_i}V\qquad(V\subset\mathbb R^{n+1}).
\end{equation}

\subsection{Local transversality for nondegenerate pieces}

\begin{lemma}[Uniform Brascamp--Lieb transversality]\label{lem:BLstability}
For every $\nu>0$ there exists $\varepsilon_{\rm BL}=\varepsilon_{\rm BL}(n,\nu)>0$ with the following property. Suppose
\begin{equation}
\|\kappa-\gamma_{\rm mo}\|_{C^n([-1,1])}\leq\varepsilon_{\rm BL},
\qquad
\gamma_{\rm mo}(s):=(s,s^2,\ldots,s^n),
\end{equation}
and $s_1,\ldots,s_{n+1}\in [-1,1]$ satisfy $|s_i-s_j|\geq\nu$ for $i\neq j$. If $\pi_{\kappa,s_i}$ denotes orthogonal projection onto $V_\kappa(s_i)$, then
\begin{equation}
\dim V\leq\frac12\sum_{i=1}^{n+1}\dim\pi_{\kappa,s_i}V
\end{equation}
for every subspace $V\subset\mathbb R^{n+1}$.
\end{lemma}

\begin{proof}
First take $\kappa=\gamma_{\rm mo}$. Denote the space of polynomials of degree at most $n$ by $\mathcal{P}_n$. If $\dim V=k$, identify $V$ with a $k$-dimensional subspace $\mathcal U\subset\mathcal P_n$ by $x\longmapsto Q_x(s):=x\cdot(\gamma_{\rm mo}(s),1)$, and put
\begin{equation}
r_i:=\operatorname{rank}\bigl(Q\mapsto(Q(s_i),Q'(s_i)):\mathcal U\to\R^2\bigr)
=\dim\pi_{\gamma_{\rm mo},s_i}V.
\end{equation}
For $k=1$, the points $s_i$ with $r_i=0$ are double zeros of a nonzero polynomial of degree at most $n$. Since a polynomial $Q_x$ has degree at most $n$, there are at most $\lfloor n/2\rfloor$ such points, and hence $\sum_i r_i\geq n+1-\lfloor n/2\rfloor\geq2=2k$.

Now suppose $k\geq2$. Choose a basis $Q_1,\ldots,Q_k$ of $\mathcal U$ whose degrees are strictly increasing. Its Wronskian
\begin{equation}
\mathcal W_{\mathcal U}(s):=\det\bigl(Q_a^{(b-1)}(s)\bigr)_{a,b=1}^{k}
\end{equation}
is a nonzero polynomial. If the degrees are $0\leq d_1<\ldots<d_k\leq n$, then the leading term computation gives
\begin{equation}\label{eq:appB-jet-wronskian-degree}
\deg\mathcal W_{\mathcal U}\leq\sum_{a=1}^{k}d_a-\frac{k(k-1)}2\leq\sum_{a=n+1-k}^{n}a-\frac{k(k-1)}2=k(n+1-k).
\end{equation}

For each $i$, choose a basis $\widetilde Q_1,\ldots,\widetilde Q_k$ of $\mathcal U$ and nonnegative integers $0\leq m_1<\cdots<m_k$ such that
\begin{equation}
\widetilde Q_a(s_i+t)=c_at^{m_a}+O(t^{m_a+1}),\qquad t=s-s_i,
\qquad
c_a\neq0.
\end{equation}
Such a basis is obtained by Gaussian elimination on the coefficient vectors of $Q(s_i+t)$, with the coefficients ordered as $1,t,\ldots,t^n$. A change of basis multiplies the Wronskian by a nonzero constant. We continue to denote the Wronskian of the new basis by $\mathcal W_{\mathcal U}$.

For $j\geq0$, let $(m)_j:=m(m-1)\cdots(m-j+1)$, with $(m)_0:=1$. Multiplying the $b$-th column of the Wronskian by $t^{b-1}$ gives
\begin{equation}
t^{\binom{k}{2}}\mathcal W_{\mathcal U}(s_i+t)
=
\left(\prod_{a=1}^kc_a\right)
\det\bigl((m_a)_{b-1}\bigr)_{a,b=1}^k
t^{m_1+\cdots+m_k}
+
O\bigl(t^{m_1+\cdots+m_k+1}\bigr).
\end{equation}
Since the polynomials $(x)_{b-1}$ have degree $b-1$ and leading coefficient $1$,
\begin{equation}
\det\bigl((m_a)_{b-1}\bigr)_{a,b=1}^k
=
\prod_{1\leq a<b\leq k}(m_b-m_a)\neq0.
\end{equation}
For a nonzero polynomial $P$, let $\operatorname{ord}_{s_i}P$ denote the largest integer $d\geq0$ such that $(s-s_i)^d$ divides $P(s)$. Consequently,
\begin{equation}\label{eq:appB-jet-wronskian-vanishing-sequence}
\operatorname{ord}_{s_i}\mathcal W_{\mathcal U}
=
\sum_{a=1}^km_a-\frac{k(k-1)}2.
\end{equation}

We now show
\begin{equation}\label{eq:appB-jet-wronskian-vanishing}
\operatorname{ord}_{s_i}\mathcal W_{\mathcal U}
\geq
(k-1)(2-r_i).
\end{equation}
The kernel of $Q\longmapsto\bigl(Q(s_i),Q'(s_i)\bigr)$ consists of the polynomials in $\mathcal U$ divisible by $(s-s_i)^2$. Thus $r_i$ equals the number of $m_a$'s less than $2$. If $r_i=2$, \eqref{eq:appB-jet-wronskian-vanishing} is trivial. If $r_i=1$, the smallest possible exponents are $(m_1,\cdots,m_k)=(0,2,3,\ldots,k)$, which gives $\operatorname{ord}_{s_i}\mathcal W_{\mathcal U}\geq k-1$ by \eqref{eq:appB-jet-wronskian-vanishing-sequence}. If $r_i=0$, the smallest possible exponents are $(m_1,\cdots,m_k)=(2,3,\ldots,k+1)$, which gives $\operatorname{ord}_{s_i}\mathcal W_{\mathcal U}\geq2k$ by \eqref{eq:appB-jet-wronskian-vanishing-sequence}. Hence, \eqref{eq:appB-jet-wronskian-vanishing} holds in every case.

Combining \eqref{eq:appB-jet-wronskian-degree} and \eqref{eq:appB-jet-wronskian-vanishing} gives
\begin{equation}
(k-1)\left(2n+2-\sum_{i=1}^{n+1}r_i\right)\leq k(n+1-k)
\end{equation}
because $\sum_{i=1}^{n+1}\operatorname{ord}_{s_i}\mathcal W_{\mathcal U}\leq\deg\mathcal W_{\mathcal U}$. It follows that
\begin{equation}\label{eq:jet-dim}
\sum_{i=1}^{n+1}r_i\geq2n+2-\frac{k}{k-1}(n+1-k)=2k+(n+1-k)\frac{k-2}{k-1}\geq2k.
\end{equation}

The inequality \eqref{eq:jet-dim} is stable under a sufficiently small $C^n$ perturbation with the separation parameter fixed. Otherwise, there would be curves $\kappa_m\to\gamma_{\rm mo}$ in $C^n([-1,1])$, $\nu$-separated tuples $s_1^{(m)},\ldots,s_{n+1}^{(m)}$, and subspaces $V_m$ for which the dimension condition fails. After passing to a subsequence, we may assume that all $V_m$'s have the same dimension $k\in\{1,\ldots,n+1\}$. Hence
\begin{equation}\label{eq:appB-local-failure}
\sum_{i=1}^{n+1}\dim\pi_{\kappa_m,s_i^{(m)}}V_m\leq2k-1.
\end{equation}
The compactness of $[-1,1]^{n+1}$ and $\operatorname{Gr}(k,n+1)$ gives a further subsequence with $s_i^{(m)}\to s_i$ and $V_m\to V$. The limiting points $s_i$ remain distinct because the tuples are $\nu$-separated. After restriction to the corresponding $k$-planes, the two defining linear functionals for $\pi_{\kappa_m,s_i^{(m)}}$ converge to those for $\pi_i^0:=\pi_{\gamma_{\rm mo},s_i}$. Rank is lower semicontinuous under matrix convergence, and therefore
\begin{equation}
\sum_{i=1}^{n+1}\dim\pi_i^0V\leq\liminf_{m\to\infty}\sum_{i=1}^{n+1}\dim\pi_{\kappa_m,s_i^{(m)}}V_m\leq2k-1.
\end{equation}
This contradicts \eqref{eq:jet-dim} and proves the existence of $\varepsilon_{\rm BL}$.
\end{proof}

Choosing $\varepsilon_{\rm loc}$ much smaller than $\varepsilon_{\rm BL}(n,\sim K^{-1})$, Lemma~\ref{lem:BLstability} applies uniformly to every $\kappa\in\mathfrak K$. The dimension condition is invariant under the fixed invertible changes of the coordinates used in the normalization.

\subsection{Stability of Brascamp--Lieb}

Recall that the intervals in Proposition~\ref{0602.prop31} have length $K^{-1}$ and mutual separation comparable to $K^{-1}$. Consider the following collection of Brascamp--Lieb data:
\begin{equation}
    \mathscr K_{\rm BL}:=\left\{\left(L_{\kappa,\rho_i,s_i}\right)_{i=1}^{n+1}:
\kappa\in\mathfrak K,\,(\rho_i,s_i)\in [1,2]\times I_i\right\}.
\end{equation}
By Arzel\`{a}-Ascoli's theorem, the closure of $\mathscr K_{\rm BL}$ is compact. By Lemma~\ref{lem:BLstability}, every tuple in $\overline{\mathscr K_{\rm BL}}$ has a finite Brascamp--Lieb constant. Hence local boundedness of the Brascamp--Lieb constant and compactness give an open neighborhood $\mathcal U_{\rm BL}$ of $\overline{\mathscr K_{\rm BL}}$ on which the Brascamp--Lieb constants are uniformly bounded; see \cite[Theorem~1.1]{MR3783217}.

Since the curves in $\mathfrak K$ are defined with uniform bounds on a fixed neighborhood of $[-1,1]$, we may choose compact rectangles $U_i^\ast$ such that $[1,2]\times I_i\Subset U_i^\ast$ for $i=1,\dots,n+1$, while the enlarged rectangles remain pairwise $\sim K^{-1}$-separated in the second coordinate. Taking $U_i^\ast$ sufficiently close to $[1,2]\times I_i$, every tuple
\begin{equation}
    \left(L_{\kappa,\rho_i,s_i}\right)_{i=1}^{n+1},\qquad 
\kappa\in\mathfrak K,\qquad (\rho_i,s_i)\in U_i^\ast
\end{equation}
belongs to $\mathcal U_{\rm BL}$.

\subsection{Thickening the Fourier support}

Recall that the Fourier supports in Proposition~\ref{0602.prop31} lie in the $AR^{-1}$-neighborhoods of the conic pieces. For bounded $R$, the desired estimate follows from Bernstein's and H\"older's inequalities, with constants allowed to depend on $A$ and $K$. We therefore assume that $R$ is sufficiently large.

Cover $[1,2]\times I_i$ by a bounded number of sufficiently small rectangles whose closures are contained in $U_i^\ast$. On each rectangle, choose smooth orthonormal vectors $\nu_{i,1}(u),\ldots,\nu_{i,n-1}(u)$ perpendicular to both $(\kappa(s),1)$ and $(\kappa'(s),0)$, where $u=(\rho,s)$. The inverse function theorem gives one-to-one coordinate maps
\begin{equation}
\Phi_i(u,z)
=
\rho(\kappa(s),1)
+
\sum_{j=1}^{n-1}z_j\nu_{i,j}(u),
\qquad
u=(\rho,s),
\qquad
|z|\lesssim_A R^{-1},
\end{equation}
with $|\det D_{(u,z)}\Phi_i|\sim1$ and uniformly bounded derivatives up to order $N_0-n-1$. For sufficiently large $R$, their images cover the $AR^{-1}$-neighborhood of $\Sigma_{\kappa,I_i}$. After applying a partition of unity consisting of $O(1)$ terms, it suffices to use one such coordinate map for each $i$; we relabel the resulting Fourier-localized pieces as $F_i$ and suppress the auxiliary indices.

For each fixed $z$, define the extension operator 
\begin{equation}
E_{i,z}g(x)
:=
\int e^{ix\cdot\Phi_i(u,z)}g(u)\,du,
\end{equation}
where the integration is over the parameter rectangle used to define $\Phi_i$. Fourier inversion and the change of variables $\xi=\Phi_i(u,z)$ give
\begin{equation}\label{eq:appB-F-slicing}
F_i(x)
=
\int_{|z|\lesssim_A R^{-1}}E_{i,z}g_{i,z}(x)\,dz,
\qquad
\int\|g_{i,z}\|_2^2\,dz
\lesssim
\|F_i\|_2^2,
\end{equation}
where $g_{i,z}(u)
:=
\widehat F_i\bigl(\Phi_i(u,z)\bigr)\left|\det D_{(u,z)}\Phi_i(u,z)\right|$.

For fixed $z_1,\ldots,z_{n+1}$ and $(\rho_i,s_i)\in U_i^\ast$, the maps $x\mapsto D_u\Phi_i((\rho_i,s_i),z_i)^Tx$ are $O_A(R^{-1})$-perturbations of $D_u\Phi_i((\rho_i,s_i),0)^T=L_{\kappa,\rho_i,s_i}$. By the choice of $U_i^\ast$, the tuple $(L_{\kappa,\rho_i,s_i})_{i=1}^{n+1}$ belongs to $\mathcal U_{\rm BL}$. Hence, for sufficiently large $R$, so does $(x\mapsto D_u\Phi_i((\rho_i,s_i),z_i)^Tx)_{i=1}^{n+1}$. Therefore, \cite[Theorem~1.3]{MR3783217} gives
\begin{equation}
\left\|
\prod_{i=1}^{n+1}
|E_{i,z_i}g_{i,z_i}|^{1/(n+1)}
\right\|_{L^{n+1}(B_R)}
\lesssim_{\varepsilon,A,K,n,\mathfrak K}
R^\varepsilon
\prod_{i=1}^{n+1}
\|g_{i,z_i}\|_2^{1/(n+1)}.
\end{equation}
Using \eqref{eq:appB-F-slicing}, Fubini's theorem, and Cauchy--Schwarz in each $(n-1)$-dimensional $z_i$-variable, we obtain
\begin{equation}
\left\|
\prod_{i=1}^{n+1}|F_i|^{1/(n+1)}
\right\|_{L^{n+1}(B_R)}
\lesssim
R^\varepsilon R^{-(n-1)/2}
\prod_{i=1}^{n+1}\|F_i\|_2^{1/(n+1)}.
\end{equation}
Finally, choose a smooth cutoff function $\chi_{B_R}$ satisfying
\begin{equation}
|\chi_{B_R}|\geq1\quad\text{on }B_R,
\qquad
|\chi_{B_R}|^2\lesssim W_{B_R},
\qquad
\operatorname{supp}\widehat{\chi_{B_R}}
\subset B(0,cR^{-1}).
\end{equation}
Applying the preceding estimate to $\chi_{B_R}F_i$, with $A$ replaced by another fixed constant, gives
\begin{equation}
\left\|
\prod_{i=1}^{n+1}|F_i|^{1/(n+1)}
\right\|_{L^{n+1}(B_R)}
\lesssim
R^\varepsilon R^{-(n-1)/2}
\prod_{i=1}^{n+1}
\|F_i\|_{L^2(W_{B_R})}^{1/(n+1)}.
\end{equation}
Passing to $L_\#^{n+1}$ and $L^2(w_{B_R})$ norms cancels the factor $R^{-(n-1)/2}$ and proves Proposition~\ref{0602.prop31}.

\section{Proof of Theorem~\ref{0524.thm54}}\label{sec:broad-narrow}
We first prove Theorem~\ref{0524.thm54} for $4<p<n+1$ by a broad-narrow analysis. The narrow part is bounded by rescaling. The broad part is controlled directly by the measure upper bound in each $R^{1/n}$-cube supplied by the Katz--Tao $(1,n,C)$ condition and the rescaled multilinear restriction estimate. Then we prove Theorem~\ref{0524.thm54} for $n+1\le p<n^2+n-2$ by interpolation.

\subsection{Setup}
Let
\begin{equation}\label{eq:ms-cube-localization}
 E_R=\bigcup_{Q\in\mathcal D_1(E):Q\cap B_R\neq\varnothing}Q.
\end{equation}
Let $\mathcal Q_{R^{1/n}}:=\mathcal D_{R^{1/n}}(E_R)$ be the set of all $R^{1/n}$-cubes that meet $E_R$.

\begin{lemma}
\label{lem:ms-density-partition}
Let $E$ be a Katz--Tao $(1,n,C)$-set, and let $E_R$ be the set in \eqref{eq:ms-cube-localization}. For any $Q\in\mathcal Q_{R^{1/n}}$,
\begin{equation}\label{eq:ms-cube-measure-bound}
 \frac{|E_R\cap Q|}{|Q|}\lesssim_{n,C} R^{-1/n}.
\end{equation}
\end{lemma}

\begin{proof}
Let $\mathcal U_Q$ be the unit cubes occurring in \eqref{eq:ms-cube-localization} that meet $Q$. Since these cubes have unit side length, every $U\in\mathcal U_Q$ is contained in the fixed enlargement $Q^+:=Q+[-2,2]^{n+1}$ and contains a point of $E\cap Q^+$. Hence
\begin{equation}
 |E_R\cap Q|\le\#\mathcal U_Q\le|E\cap Q^+|_1.
\end{equation}
Because $R \geq 1$, the box $Q^+$ is covered by $O_n(1)$ cubes of $\mathcal D_{R^{1/n}}$. Applying the $(1,n,C)$ condition to those cubes gives $|E\cap Q^+|_1\lesssim_{n,C} (R^{1/n})^n$. Since $|Q|=(R^{1/n})^{n+1}$, this proves \eqref{eq:ms-cube-measure-bound}.
\end{proof}

Choose $\beta_{\rm bn}$ so that
\begin{equation}
 0<\beta_{\rm bn}<\frac14\left(\frac12-\frac2p\right).
\end{equation}
Choose $0<\varepsilon_{\rm bn}<1$, which will be fixed below in the estimate for the broad term. Let $\Delta$ be dyadic with
\begin{equation}\label{eq:ms-Delta-choice}
 R^{-\varepsilon_{\rm bn}^2/n}\leq \Delta<2R^{-\varepsilon_{\rm bn}^2/n}.
\end{equation}
Assume that $R$ is large enough so that $\Delta<1$.

For every dyadic interval $I\subset [-\frac12,\frac12]$ of length $\geq\Delta$, define
\begin{equation}\label{eq:ms-dyadic-pieces}
 F_I:=\sum_{s_\theta\in I}F_\theta.
\end{equation}
Apply Lemma~\ref{lem:ms-abstract-bn} on $X=E_R$, with $q=p$, $k=n+1$, $\delta=\Delta$, $\beta_{\rm abs}=\beta_{\rm bn}$, and $F_I$. The lemma supplies the narrow and broad scales $\delta_{\rm nar}$ and $\delta_{\rm br}$ satisfying \eqref{eq:ms-stopping-scales}, and the constant $C_{\beta_{\rm bn},n+1}$ that determines the length of the intervals $J$ in the broad sum, such that the following inequality holds.
\begin{equation}\label{eq:ms-master-decomposition}
\begin{split}
 \left\|\sum_\theta F_\theta\right\|_{L^p(E_R)}
 &\lesssim
 \Delta^{-\beta_{\rm bn}}\left(\sum_{|I|=\delta_{\rm nar}}
\|F_I\|_{L^p(E_R)}^p\right)^{1/p}
\\&+ \Delta^{-\beta_{\rm bn}}\left(
\sum_{|J|=C_{\beta_{\rm bn},n+1}\delta_{\rm br}}
\ \sum_{\mathbf I\in
\operatorname{Sep}_{n+1}(J;\delta_{\rm br})}
\left\|
\prod_{j=1}^{n+1}|F_{I_j}|^{\frac1{n+1}}
\right\|_{L^p(E_R)}^p
\right)^{1/p}.
\end{split}
\end{equation}
We estimate the two terms on the right-hand side separately. We call these terms narrow and broad, respectively.
\subsection{The narrow term}
We begin by estimating the narrow term. For $|I|=\delta_{\rm nar}$, the inequalities \eqref{eq:ms-stopping-scales} and \eqref{eq:ms-Delta-choice} give
\begin{equation}\label{eq:ms-narrow-scale}
 R|I|^n\sim_{\beta_{\rm bn},n} R\Delta^n\sim \frac{R}{R^{\varepsilon_{\rm bn}^2}}.
\end{equation}
Since $0<\varepsilon_{\rm bn}<1$, $R|I|^n\geq1$ if $R$ is sufficiently large. Applying Lemma~\ref{lem:ms-rescaled-envelope} to the function $F_I$ with $\rho=|I|=\delta_{\mathrm{nar}}$ on $2B_R$, with $q=p$ and $\varepsilon>0$ chosen so that
\begin{equation}\label{eq:ms-rescaling-loss-choice}
 0<\varepsilon\leq
 \frac{\varepsilon_{\rm bn}^2\left(\frac12-\frac2p-\beta_{\rm bn}\right)}{2n},
\end{equation}
we obtain
\begin{equation}\label{eq:ms-narrow-one-I}
 \|F_I\|_{L^p(E_R)}
 \lesssim_{\beta_{\rm bn},p,n}
 \big(\frac{R}{R^{\varepsilon_{\rm bn}^2}} \big)^{\frac1n\left(\frac12-\frac2p\right)+\varepsilon}
 \left\|\left(\sum_{s_\theta\in I}|F_\theta|^{p/2}\right)^{2/p}\right\|_{L^p(W_{B_R^{(n+1)}})}.
\end{equation}
The intervals $I$ at the scale $\delta_{\rm nar}$ are disjoint, and hence
\begin{equation}\label{eq:ms-narrow-square-sum}
 \sum_I
 \left(\sum_{s_\theta\in I}|F_\theta|^{p/2}\right)^2
 \leq
 \left(\sum_\theta|F_\theta|^{p/2}\right)^2.
\end{equation}
It follows that
\begin{equation}\label{eq:ms-narrow-before-exponents}
 \Delta^{-\beta_{\rm bn}} \left(\sum_{|I|=\delta_{\rm nar}}
\|F_I\|_{L^p(E_R)}^p\right)^{1/p}
 \lesssim
 \Delta^{-\beta_{\rm bn}}\big(\frac{R}{R^{\varepsilon_{\rm bn}^2}} \big)^{\frac1n\left(\frac12-\frac2p\right)+\varepsilon}
 \left\|\left(\sum_{\theta}|F_\theta|^{p/2}\right)^{2/p}\right\|_{L^p(W_{B_R^{(n+1)}})}.
\end{equation}
Since $\Delta^{-1}\sim R^{\varepsilon_{\rm bn}^2/n}$, the exponent of $R$ in the coefficient on the right is at most
\begin{equation}
\frac1n\left(\frac12-\frac2p\right)
-\frac{\varepsilon_{\rm bn}^2\left(\frac12-\frac2p-\beta_{\rm bn}\right)}{n}
+\varepsilon
\leq
\frac1n\left(\frac12-\frac2p\right)
-\frac{\varepsilon_{\rm bn}^2\left(\frac12-\frac2p-\beta_{\rm bn}\right)}{2n}.
\end{equation}
Define $c_{\rm nar}:=
 \frac{\varepsilon_{\rm bn}^2}{2n}\left(\frac12-\frac2p-\beta_{\rm bn}\right)>0$. Then we obtain
\begin{equation}\label{eq:ms-narrow-final} \Delta^{-\beta_{\rm bn}} 
\left(\sum_{|I|=\delta_{\rm nar}}
\|F_I\|_{L^p(E_R)}^p\right)^{1/p}
 \lesssim
 R^{\frac1n\left(\frac12-\frac2p\right)-c_{\rm nar}}
 \left\|\left(\sum_{\theta}|F_\theta|^{p/2}\right)^{2/p}\right\|_{L^p(W_{B_R^{(n+1)}})}.
\end{equation}
This bounds the narrow term on the right-hand side of \eqref{eq:ms-master-decomposition}.

\subsection{Rescaled multilinear estimates}

We now estimate the broad term under the assumption $4<p<n+1$. Let $\rho_{\rm bn}:=C_{\beta_{\rm bn},n+1}\delta_{\rm br}$ denote the common length of the intervals $J$ in the broad sum. If $\rho_{\rm bn}>1$, then this sum is empty by Lemma~\ref{lem:ms-abstract-bn}. Hence we assume $0<\rho_{\rm bn}\leq1$. Since $\rho_{\rm bn}\gtrsim \Delta$, rescaling an $R^{1/n}$-cube associated with such an interval $J$ gives the scale
\begin{equation}\label{eq:ms-broad-parent-scale-lower}
 R_J:=R^{1/n}\rho_{\rm bn}^n
 \gtrsim_{\beta_{\rm bn},n}(R^{1/n})\Delta^n
 \sim R^{1/n-\varepsilon_{\rm bn}^2}.
\end{equation}
We choose $\varepsilon_{\rm bn}^2<1/n$. Thus, for all sufficiently large $R$, one has $R_J\geq C_{\beta_{\rm bn},n+1}$.
\medskip

We rescale the multilinear restriction estimate in Proposition~\ref{0602.prop31} to intervals of length $\rho_{\rm bn}$. 

\begin{lemma}
\label{lem:ms-selected-parent-multilinear}
Fix $\eta\in\mathfrak E$.
Let $Q$ be an $R^{1/n}$-cube, let $J\subseteq[-\frac12,\frac12]$ be an interval of length $\rho_{\rm bn}$, and let $(I_1,\ldots,I_{n+1})\in\operatorname{Sep}_{n+1}(J;\delta_{\rm br})$.

For each $1\leq\ell\leq n+1$, let $(H_{\ell,\theta})_{\theta\in\mathcal G_{R^{-1/n}}}$ be admissible at scale $R$ with respect to the same curve $\eta$, in the sense of Definition~\ref{def:ms-admissible-cone-blocks}.
Thus, setting
\begin{equation}\label{eq:ms-selected-family-support}
 H_{\ell,I_\ell}
 :=\sum_{s_\theta\in I_\ell}H_{\ell,\theta},
 \qquad
 \operatorname{supp}\widehat H_{\ell,\theta}
 \subset \mathfrak B_\eta(s_\theta;R^{-1/n}),
 \qquad 1\leq\ell\leq n+1,
\end{equation}
we have, for every $\varepsilon>0$,
\begin{align}
 \left\|\prod_{\ell=1}^{n+1}|H_{\ell,I_\ell}|^{\frac1{n+1}}
 \right\|_{L^{n+1}(Q)}
 \lesssim_{\beta_{\rm bn},\varepsilon,n}
 \rho_{\rm bn}^{-n^{2}/2}R^{\varepsilon} |Q|^{1/(n+1)}
 \label{eq:ms-selected-parent-multilinear}
 \prod_{\ell=1}^{n+1}
 \left(
 \frac1{|Q|}\int |H_{\ell,I_\ell}|^2
 W_{Q}
 \right)^{1/(2n+2)}.
 \notag
\end{align}
\end{lemma}

\begin{proof}
Let $s_J$ be the center of $J$, and use the map $\mathsf A_J$ from Lemma~\ref{lem:ms-parent-normalization}. Set
\begin{equation}
 \widetilde H_\ell(y):=H_{\ell,I_\ell}(\mathsf A_J^Ty),
 \qquad
 \widetilde Q:=\mathsf A_J^{-T}Q.
\end{equation}
The rescaled intervals $(I_\ell-s_J)/\rho_{\rm bn}$ have length $C_{\beta_{\rm bn},n+1}^{-1}$ and a fixed positive mutual separation. Thus, after a fixed enlargement, they satisfy the separation condition in Proposition~\ref{0602.prop31} with $K=C_{\beta_{\rm bn},n+1}$. By Lemma~\ref{lem:ms-parent-normalization}(2), every $\widetilde H_\ell$ has Fourier support in a fixed enlargement of an $R_J^{-1}$-neighborhood of the corresponding normalized conic piece.

By \eqref{eq:ms-AJ-definition}, the set $\widetilde Q$ is comparable to an ellipsoid whose principal axes have lengths comparable to
\begin{equation*}
R^{1/n},\ R^{1/n}\rho_{\rm bn},\ R^{1/n}\rho_{\rm bn}^2,\ldots,R^{1/n}\rho_{\rm bn}^n.
\end{equation*}
Thus $\widetilde Q$ is covered by a family $\mathscr B_J$ of balls of radius comparable to $R_J=R^{1/n}\rho_{\rm bn}^n$, with bounded overlap and
\begin{equation}\label{eq:ms-parent-cover-number}
 \#\mathscr B_J\lesssim_n\rho_{\rm bn}^{-\frac{n(n+1)}{2}}.
\end{equation}
For each $B\in\mathscr B_J$, Proposition~\ref{0602.prop31} gives
\begin{align*}
 \left\|\prod_{\ell=1}^{n+1}|\widetilde H_\ell|^{1/(n+1)}
 \right\|_{L^{n+1}(B)}
 \lesssim{}&
 R_J^{\varepsilon+1-(n+1)/2}
 \prod_{\ell=1}^{n+1}
 \left(\int|\widetilde H_\ell|^2W_B\right)^{1/(2n+2)}.
\end{align*}
After enlarging $\widetilde Q$ by a fixed factor, every weight $W_B$ is bounded by a constant multiple of the weight $y\mapsto W_{CQ}(\mathsf A_J^Ty)$. Indeed, if $y_B$ is the center of $B$ and $y_Q$ the center of $\widetilde Q$, then $\mathsf A_J^T(y_B-y_Q)\in C(Q-\mathsf A_J^T y_Q)$, while $\|\mathsf A_J\|/R^{1/n}\lesssim( R^{1/n}\rho_{\rm bn}^n)^{-1}=R_J^{-1}$; hence
\begin{equation}
 1+\frac{|\mathsf A_J^T(y-y_Q)|}{R^{1/n}}
 \lesssim 1+\frac{|y-y_B|}{R_J}.
\end{equation}
Summing the $(n+1)$-th powers over the cover and using $(\#\mathscr B_J)^{1/(n+1)}\leq(\#\mathscr B_J)^{1/2}$, we obtain
\begin{align*}
 \left\|\prod_{\ell=1}^{n+1}|\widetilde H_\ell|^{\frac1{n+1}}
 \right\|_{L^{n+1}(\widetilde Q)}
 \lesssim{}
 \rho_{\rm bn}^{-\frac{n(n+1)}4}R_J^{\varepsilon+1-\frac{n+1}2}
 \prod_{\ell=1}^{n+1}
 \left(
 \int|\widetilde H_\ell(y)|^2
 W_{Q}(\mathsf A_J^Ty)\,dy
 \right)^{1/(2n+2)}
\end{align*}
since $W_{CQ}\lesssim_C W_Q$. Changing variables back to $x=\mathsf A_J^Ty$ contributes $|\det \mathsf A_J|^{\frac1{n+1}-\frac12}$. Since $R_J=R^{1/n}\rho_{\rm bn}^n$ and $|\det \mathsf A_J|\sim\rho_{\rm bn}^{-\frac{n(n+1)}2}$,
\begin{align*}
\rho_{\rm bn}^{-\frac{n(n+1)}4}R_J^{\varepsilon+1-\frac{n+1}2}
 |\det \mathsf A_J|^{\frac1{n+1}-\frac12}
 &\lesssim
 (R^{1/n})^{\varepsilon+1-\frac{n+1}2}\rho_{\rm bn}^{-n^2/2+n\varepsilon}\\
 &\lesssim_{\varepsilon,n}
(R^{1/n})^{\varepsilon+1-\frac{n+1}2}\rho_{\rm bn}^{-n^{2}/2},
\end{align*}
where the last inequality uses $0<\rho_{\rm bn}\leq1$. Finally,
\begin{equation}
 (R^{1/n})^{1-(n+1)/2}
 \prod_{\ell=1}^{n+1}
 \left(\int|H_{\ell,I_\ell}|^2W_{Q}\right)^{1/(2n+2)}
 =|Q|^{1/(n+1)}
 \prod_{\ell=1}^{n+1}
 \left(\frac1{|Q|}\int|H_{\ell,I_\ell}|^2W_{Q}\right)^{1/(2n+2)}.
\end{equation}
This proves the desired inequality.
\end{proof}

For the broad term in Theorem~\ref{0524.thm54}, we use the following consequence of Lemma~\ref{lem:ms-selected-parent-multilinear}.

\begin{corollary}
\label{cor:ms-selected-direct}
Under the hypotheses of Lemma~\ref{lem:ms-selected-parent-multilinear}, assume also that
\begin{equation}\label{eq:ms-selected-pointwise-dominance}
 |H_{\ell,\theta}(x)|\leq|F_\theta(x)|
 \qquad(1\leq\ell\leq n+1).
\end{equation}
For every $\varepsilon>0$ and every measurable $Y\subset Q$,
\begin{equation}\label{eq:ms-selected-direct}
 \left\|\prod_{\ell=1}^{n+1}|H_{\ell,I_\ell}|^{\frac1{n+1}}
 \right\|_{L^p(Y)}
 \lesssim_{\varepsilon}
 \rho_{\rm bn}^{-n^{2}/2}(R^{\frac1n})^{\frac12-\frac2p+\varepsilon}
 \left(\frac{|Y|}{|Q|}\right)^{\frac1p-\frac1{n+1}}
 \left\|\left(\sum_{s_\theta\in J}|F_\theta|^{p/2}\right)^{2/p}\right\|_{L^p(W_{Q})}.
\end{equation}
\end{corollary}

\begin{proof}
Since $p<n+1$, H\"older's inequality gives
\begin{equation}
 \left\|\prod_{\ell=1}^{n+1}|H_{\ell,I_\ell}|^{\frac1{n+1}}
 \right\|_{L^p(Y)}
 \leq |Y|^{\frac1p-\frac1{n+1}}
 \left\|\prod_{\ell=1}^{n+1}|H_{\ell,I_\ell}|^{\frac1{n+1}}
 \right\|_{L^{n+1}(Q)}.
\end{equation}
Note that $\operatorname{supp}\widehat H_{\ell,\theta}\subset\mathfrak B_\eta(s_\theta;R^{-1/n})$ by \eqref{eq:ms-selected-family-support}. For some small $c>0$, the family $\{\mathfrak B_\eta(s_\theta;R^{-1/n})+B(0,cR^{-1/n})\}_\theta$ has bounded overlap. We may thus apply Lemma~\ref{lem:ms-selected-local-orthogonality} to $H_{\ell,I_\ell}=\sum_{s_\theta\in I_\ell}H_{\ell,\theta}$ and then apply \eqref{eq:ms-selected-pointwise-dominance} to obtain
\begin{equation}
 \frac1{|Q|}\int |H_{\ell,I_\ell}|^2
 W_{Q}\lesssim
 \frac1{|Q|}\int
 \sum_{s_\theta\in J}|H_{\ell,\theta}|^2W_{Q}\lesssim
 \frac1{|Q|}\int
 \sum_{s_\theta\in J}|F_\theta|^2W_{Q},\qquad 1\leq \ell\leq n+1.
\end{equation}

By Lemma~\ref{lem:ms-selected-parent-multilinear}, it suffices to show that
\begin{equation}\label{eq:L2-energy-ell-p/2-bound}
 \left(\frac1{|Q|}\int
 \sum_{s_\theta\in J}|F_\theta|^2W_{Q}\right)^{1/2}\lesssim
 (R^{\frac1n})^{\frac12-\frac2p}|Q|^{-1/p}\left\|\left(\sum_{s_\theta\in J}|F_\theta|^{p/2}\right)^{2/p}\right\|_{L^p(W_{Q})}.
\end{equation}
There are $O(R^{1/n})$ intervals in $\mathcal G_{R^{-1/n}}$ whose centers lie in $J$, and therefore
\begin{equation}
 \left(\sum_{s_\theta\in J}|F_\theta|^2\right)^{1/2}
 \lesssim (R^{\frac1n})^{\frac12-\frac2p}\left(\sum_{s_\theta\in J}|F_\theta|^{p/2}\right)^{2/p}.
\end{equation}
H\"older's inequality and $\int W_{Q}\lesssim|Q|$ imply \eqref{eq:L2-energy-ell-p/2-bound}.
\end{proof}

\subsection{Estimate for the broad term}
Now we estimate the broad term.

\begin{proposition}
\label{prop:ms-geometric-broad}
There exists $c_{\rm br}=c_{\rm br}(p,n)>0$ such that, after choosing $\varepsilon_{\rm bn}>0$ sufficiently small,
\begin{equation}\label{eq:ms-broad-final}
\begin{split}
\Delta^{-\beta_{\rm bn}}\Bigg(
\sum_{|J|=C_{\beta_{\rm bn},n+1}\delta_{\rm br}}
\sum_{\mathbf I\in
\operatorname{Sep}_{n+1}(J;\delta_{\rm br})}
&\left\|
\prod_{j=1}^{n+1}|F_{I_j}|^{\frac1{n+1}}
\right\|_{L^p(E_R)}^p
\Bigg)^{1/p}
\\& \lesssim
 R^{\frac1n\left(\frac12-\frac2p\right)-c_{\rm br}}
 \left\|\left(\sum_\theta|F_\theta|^{p/2}\right)^{2/p}\right\|_{L^p(W_{B_R^{(n+1)}})}.
 \end{split}
\end{equation}
\end{proposition}

\begin{proof}
Fix an interval $J$, a separated tuple $(I_1,\ldots,I_{n+1})\in\operatorname{Sep}_{n+1}(J;\delta_{\rm br})$, and an $R^{1/n}$-cube $Q\in\mathcal Q_{R^{1/n}}$. By Lemma~\ref{lem:ms-density-partition},
\begin{equation}
 \frac{|E_R\cap Q|}{|Q|}\lesssim_{n,C}R^{-1/n}.
\end{equation}
Applying Corollary~\ref{cor:ms-selected-direct} with $H_{\ell,\theta}=F_\theta$ and $Y=E_R\cap Q$, we obtain, for every $\varepsilon>0$,
\begin{equation}\label{eq:ms-broad-density-local}
 \left\|\prod_{\ell=1}^{n+1}|F_{I_\ell}|^{1/(n+1)}
 \right\|_{L^p(E_R\cap Q)}
 \lesssim
 \rho_{\rm bn}^{-n^2/2}
 (R^{\frac1n})^{\frac12-\frac2p-(\frac1p-\frac1{n+1})+\varepsilon}
 \left\|\left(\sum_{s_\theta\in J}|F_\theta|^{p/2}\right)^{2/p}\right\|_{L^p(W_{Q})}.
\end{equation}

Raise \eqref{eq:ms-broad-density-local} to the $p$-th power and sum over $Q\in\mathcal Q_{R^{1/n}}$, using $\sum_{Q\in\mathcal Q_{R^{1/n}}}W_Q\lesssim W_{B_R^{(n+1)}}$. Note that
\begin{equation}\label{eq:ms-parent-square-sum}
 \sum_J\left(\sum_{s_\theta\in J}|F_\theta|^{p/2}\right)^2
 \lesssim
 \left(\sum_\theta|F_\theta|^{p/2}\right)^2
\end{equation}
because for each $J$ there are $O_{\beta_{\rm bn},n}(1)$ tuples in $\operatorname{Sep}_{n+1}(J;\delta_{\rm br})$ and the intervals $J$ have bounded overlap. Since $\rho_{\rm bn}\gtrsim_{\beta_{\rm bn},n}\Delta$ and $\Delta^{-1}\lesssim_{\beta_{\rm bn},n}R^{\varepsilon_{\rm bn}^2/n}$, the left-hand side of \eqref{eq:ms-broad-final} is bounded by
\begin{equation}\label{eq:ms-broad-density-exponents}
 (R^{\frac1n})^{\frac12-\frac2p-(\frac1p-\frac1{n+1})+\varepsilon+
 \varepsilon_{\rm bn}^2\left(\frac{n^2}{2}+\beta_{\rm bn}\right)}
 \left\|\left(\sum_\theta|F_\theta|^{p/2}\right)^{2/p}\right\|_{L^p(W_{B_R^{(n+1)}})}.
\end{equation}
Choose $\varepsilon_{\rm bn}>0$ so small that
\begin{equation}
 \varepsilon_{\rm bn}^2<\frac1n,
 \qquad
 \varepsilon_{\rm bn}^2\left(\frac{n^2}{2}+\beta_{\rm bn}\right)\leq\frac{1}{4}\left(\frac1p-\frac1{n+1}\right),
\end{equation}
and then choose $0<\varepsilon\leq (\frac1p-\frac1{n+1})/4$. The exponent of $R$ in \eqref{eq:ms-broad-density-exponents} is then at most
\begin{equation}
 \frac1n \Bigg( \frac12-\frac2p-\frac{1}{2}\left(\frac1p-\frac1{n+1}\right) \Bigg).
\end{equation}
Thus we obtain \eqref{eq:ms-broad-final} with
\begin{equation}
 c_{\rm br}:=\frac1{2n}\left(\frac1p-\frac1{n+1}\right)>0.
\end{equation}
This finishes the proof.
\end{proof}

Combining \eqref{eq:ms-master-decomposition}, \eqref{eq:ms-narrow-final}, and \eqref{eq:ms-broad-final}, we obtain
\begin{equation}\label{eq:ms-small-p-final-estimate}
 \left\|\sum_\theta F_\theta\right\|_{L^p(E_R)}
 \lesssim_{p,n,C}
 R^{\frac1n\left(\frac12-\frac2p\right)-c_0}
 \left\|\left(\sum_\theta|F_\theta|^{p/2}\right)^{2/p}\right\|_{L^p(W_{B_R^{(n+1)}})}
\end{equation}
with $c_0:=\min\{c_{\rm nar},c_{\rm br}\}>0$. This completes the proof of Theorem~\ref{0524.thm54} for the case $4<p<n+1$.

\subsection{Interpolation for the remaining exponent range}
\label{subsec:ms-large-p}

We now prove Theorem~\ref{0524.thm54} for $n+1\leq p<n^2+n-2$, using Lemma~\ref{lem:ms-large-p-admissible-interpolation}. Fix $4<p_\#<n+1$, for instance $p_\#=(n+5)/2$, and recall $p_n:=n^2+n-2$. The interpolation calculation below only requires $p_\#<p<p_n$; this overlap with the small-$p$ range will be used in Remark~\ref{rem:ms-gain-local-uniformity} when $n=5$.

Since $E_R\subseteq 2B_R^{(n+1)}$ for sufficiently large $R$, we choose a Schwartz function $\chi$ such that 
\begin{equation}\label{eq:ms-large-p-common-localizer}
 |\chi|\geq1\ \text{on }2B_R,
 \qquad
 \operatorname{supp}\widehat\chi\subset B(0,cR^{-1}),
 \qquad
 |\chi(x)|^t\lesssim W_{B_R^{(n+1)}}(x)
 \quad(p_\#\leq t\leq p_n),
\end{equation}
where we take $c$ sufficiently small so that 
\begin{equation}
 \operatorname{supp}\widehat K_\theta\subset\Pi_R^\varsigma(\theta)+B(0,cR^{-1})\subset\mathfrak B_\eta(s_\theta;R^{-1/n})
\end{equation}
for $K_\theta:=\chi F_\theta$. Thus $(K_\theta)_\theta$ is admissible at scale $R$.

We verify the two endpoint hypotheses of Lemma~\ref{lem:ms-large-p-admissible-interpolation}. Since $4<p_\#<n+1$, every admissible family $\mathbf G=(G_\vartheta)_{\vartheta\in\mathcal G_{R^{-1/n}}}$ at scale $R$ satisfies
\begin{equation}\label{eq:ms-large-p-small-endpoint-preliminary}
\left\|\sum_\vartheta G_\vartheta\right\|_{L^{p_\#}(E_R)}
\lesssim
R^{\frac1n\left(\frac12-\frac2{p_\#}\right)-c_{\rm sp}}
\left\|\left(\sum_\vartheta |G_\vartheta|^{p_\#/2}\right)^{2/p_\#}\right\|_{L^{p_\#}(W_{B_{R}^{(n+1)}})},
\end{equation}
where $c_{\rm sp}=c_{\rm sp}(p_\#,n)>0$ is the gain for the exponent $p_\#$. Choose
\begin{equation}\label{eq:ms-large-p-gamma0}
0<\gamma_0\leq
\min\left\{
c_{\rm sp},
\frac1{2n}\left(\frac12-\frac2{p_\#}\right)
\right\}.
\end{equation}
Then every admissible family $\mathbf G$ indexed by $\mathcal G_{R^{-1/n}}$ at scale $R$ satisfies
\begin{equation}\label{eq:ms-large-p-small-endpoint}
\begin{split}
\left\|\sum_\vartheta G_\vartheta\right\|_{L^{p_\#}(E_R)}
&\lesssim
R^{\frac1n\left(\frac12-\frac2{p_\#}\right)-\gamma_0}
\left\|
\left(\sum_\vartheta|G_\vartheta|^{p_\#/2}\right)^{2/p_\#}
\right\|_{L^{p_\#}(W_{B_R^{(n+1)}})}
\\
&\leq
R^{\frac1n 
\left(\frac12-\frac2{p_\#}\right)-\gamma_0}
\left\|
\left(\sum_\vartheta|G_\vartheta|^{p_\#/2}\right)^{2/p_\#}
\right\|_{L^{p_\#}}.
\end{split}
\end{equation}
At $p_n=n^2+n-2$, apply
\eqref{eq:ms-family-form-gmo} with scale $R$ to obtain
\begin{equation}
\left\|\sum_\vartheta G_\vartheta\right\|_{L^{p_n}(E_R)}
\lesssim_\varepsilon
R^{\frac1n(\frac12-\frac2{p_n})+\varepsilon}
\left\|\left(\sum_\vartheta|G_\vartheta|^{p_n/2}\right)^{2/p_n}
\right\|_{L^{p_n}}.
\end{equation}

For $p\in[n+1,n^2+n-2)$, define
\begin{equation}\label{eq:ms-large-p-lambda}
 \frac1p=\frac\lambda{p_\#}+\frac{1-\lambda}{p_n},
 \qquad
 \lambda=
 \frac{p^{-1}-p_n^{-1}}{p_\#^{-1}-p_n^{-1}}\in(0,1).
\end{equation}
Since $4<p_\#<p<p_n$, we may apply Lemma~\ref{lem:ms-large-p-admissible-interpolation} to $(K_\theta)_\theta$, with
\begin{equation}
 A_0\sim R^{\frac1n\left(\frac12-\frac2{p_\#}\right)-\gamma_0},
 \qquad
 A_1\sim R^{\frac1n\left(\frac12-\frac2{p_n}\right)+\varepsilon}.
\end{equation}
Note that \eqref{eq:ms-large-p-gamma0} implies $A_0,A_1\geq1$ for large $R$. Since $|\chi|\geq1$ on $E_R$, we obtain
\begin{align}
 \left\|\sum_\theta F_\theta\right\|_{L^p(E_R)}
 &\leq
 \left\|\sum_\theta
 K_\theta\right\|_{L^p(E_R)}\notag\\
 &\lesssim_\varepsilon
 R^{O(\varepsilon)}
 \left(R^{\frac1n\left(\frac12-\frac2{p_\#}\right)-\gamma_0}\right)^\lambda
 \left(R^{\frac1n\left(\frac12-\frac2{p_n}\right)}\right)^{1-\lambda}\left\|\left(\sum_\theta |K_\theta|^{p/2}\right)^{2/p}\right\|_{L^p}.
 \label{eq:ms-large-p-after-admissible-interpolation}
\end{align}
Since $K_\theta=\chi F_\theta$, \eqref{eq:ms-large-p-common-localizer} gives
\begin{equation}
 \left\|\left(\sum_\theta |K_\theta|^{p/2}\right)^{2/p}\right\|_{L^p}
 =\left\||\chi|\left(\sum_\theta |F_\theta|^{p/2}\right)^{2/p}\right\|_{L^p}
 \lesssim
 \left\|\left(\sum_\theta |F_\theta|^{p/2}\right)^{2/p}\right\|_{L^p(W_{B_R^{(n+1)}})}.
\end{equation}
Substitution into \eqref{eq:ms-large-p-after-admissible-interpolation} yields
\begin{equation}\label{eq:ms-large-p-before-loss-choice}
 \left\|\sum_\theta F_\theta\right\|_{L^p(E_R)}
 \lesssim_\varepsilon R^{O(\varepsilon)}
 R^{\frac1n\left(\frac12-\frac2p\right)-\lambda\gamma_0}
 \left\|\left(\sum_\theta |F_\theta|^{p/2}\right)^{2/p}\right\|_{L^p(W_{B_R^{(n+1)}})},
\end{equation}
because
\begin{equation}
 \lambda\left(\frac12-\frac2{p_\#}\right)+(1-\lambda)\left(\frac12-\frac2{p_n}\right)=\frac12-\frac2p.
\end{equation}
Taking $\varepsilon>0$ sufficiently small depending on $p$ and $n$, we obtain
\begin{equation}\label{eq:ms-large-p-final-R}
 \left\|\sum_\theta F_\theta\right\|_{L^p(E_R)}
 \lesssim_{p,n,C}
 R^{\frac1n\left(\frac12-\frac2p\right)-\lambda\gamma_0/2}
 \|\left(\sum_\theta |F_\theta|^{p/2}\right)^{2/p}\|_{L^p(W_{B_R^{(n+1)}})}.
\end{equation}
This completes the proof of Theorem~\ref{0524.thm54} with
\begin{equation}
 c(p,n)=\frac{n\lambda(p)\gamma_0}{2}>0.
\end{equation}

\medskip
\paragraph{Verification of Remark~\ref{rem:ms-gain-local-uniformity}}
If $n\ge6$, then $2n-4>n+1$, and the formula $c(p,n)=n\lambda(p)\gamma_0/2$ in \eqref{eq:ms-large-p-lambda}--\eqref{eq:ms-large-p-final-R} is continuous and positive on a neighborhood of $2n-4$. If $n=5$, then $2n-4=6$, and taking $p_\#=(n+5)/2$ as above yields $p_\#=5$. The proof of \eqref{eq:ms-large-p-final-R} uses Lemma~\ref{lem:ms-large-p-admissible-interpolation} and requires only $p_\#<p<p_n$; hence the same proof also applies for $5<p<6$. Its gain $c(p,5)=5\lambda(p)\gamma_0/2$ is continuous and positive on a neighborhood of $6$. This proves Remark~\ref{rem:ms-gain-local-uniformity}.

\bibliographystyle{alpha}
\bibliography{reference}

\end{document}